\documentclass[a4paper,12pt]{article}  
\usepackage{graphicx}
\usepackage{amsmath}
\usepackage{amssymb}
\usepackage{enumerate}
\usepackage{amsthm}
\usepackage[T1]{fontenc}
\usepackage[latin1]{inputenc}
\usepackage[round]{natbib}

\newtheorem{thm}{Theorem}[section]
\newtheorem{prop}[thm]{Proposition}
\newtheorem{cor}[thm]{Corollary}
\newtheorem{lemma}[thm]{Lemma}
\theoremstyle{remark}
\newtheorem{remark}[thm]{Remark}
\newtheorem{example}[thm]{Example}
\theoremstyle{definition}

\newcommand{\pr}[2]{\mathbb{P}_{#1}\left(#2\right)}

\newcommand{\E}[2]{\mathbb{E}_{#1}\left[#2\right]}
\newcommand{\indic}[1]{1_{\left\{#1\right\}}}
\newcommand{\abs}[1]{\left\vert#1\right\vert}
\newcommand{\set}[1]{\left\lbrace #1 \right\rbrace}

\date{\today}

\begin{document}

\title{Survival probabilities of autoregressive processes}
\author{\renewcommand{\thefootnote}{\arabic{footnote}}{\sc Christoph Baumgarten}\footnotemark[1]}
\date{\today}

\footnotetext[1]{
Technische Universit\"at Berlin, Institut f\"ur Mathematik, Sekr.\ MA 7-4,
Stra{\ss}e des 17.\ Juni 136, 10623 Berlin, Germany. Email: {\sl baumgart@math.tu-berlin.de}.
}
\maketitle
\begin{abstract}
Given an autoregressive process $X$ of order $p$ (i.e.\ $X_n = a_1 X_{n-1} + \dots + a_p X_{n-p} + Y_n$ where the random variables $Y_1,Y_2,\dots$ are i.i.d.), we study the asymptotic behaviour of the probability that the process does not exceed a constant barrier up to time $N$ (survival or persistence probability). Depending on the coefficients $a_1,\dots,a_p$ and the distribution of $Y_1$, we state conditions under which the survival probability decays polynomially, faster than polynomially or converges to a positive constant. Special emphasis is put on AR(2) processes.  
\end{abstract}
\newpage

\section{Introduction}
For fixed $p \geq 1$, define $X_{n} = \sum_{k=1}^p a_k X_{n-k} + Y_{n}$, $n \geq 0$ with the convention that $X_n = 0$ for $n \leq 0$. Troughout the paper, we assume that  $(Y_n)_{n \geq 1}$ is a sequence of i.i.d.\ (nondegenerate) random variables. $(X_n)_{n \geq 1}$ is called an autoregressive process of order $p$ (AR($p$)-process). We sometimes refer to the random variables $(Y_n)_{n \ge 1}$ as innovations. Denote by $p_N(x)$ the probability that the process $X$ stays below $x$ until time $N$, i.e.\
\[
   p_N(x) := \pr{}{\sup_{n=1,\dots,N} X_n \leq x}, \quad N \geq 1, x \geq 0.
\]
We refer to $p_N$ as the survival probability up to time $N$, and we write $p_N$ instead of $p_N(0)$ in the sequel. \\
The aim of this paper is to study the asymptotic behaviour of $p_N(x)$ as $N \to \infty$. Sometimes, the problem of determining the asymptotic behaviour of $p_N(x)$ is referred to as one-sided exit or one-sided barrier problem since $p_N(x) = \pr{}{\tau_x > N}$ where $\tau_x := \inf \set{n \geq 0: X_n > x }$. Such asymptotic results are known in a number of special cases such as random walks, integrated random walks, fractional Brownian motion and AR($1$)-processes. The study of survival probabilities is motivated by several applications such as the inviscid Burgers equation (\cite{sinai:1992-a}) or zeros of random polynomials (\cite{dpsz:2002}). We refer to the recent survey of \cite{aurzada-simon:2012}, \cite{li-shao:2004} and \cite{li-shao:2005} for further information, applications and references. For instance, if $X$ is a random walk ($p=1, a_1 = 1$), it holds that $p_N(x) \sim c(x) N^{-1/2}$ if $\E{}{Y_1}=0$ and $\E{}{Y_1^2} = 1$ (see e.g.\ \cite{feller:1971-vol2}). \cite{novikov-kordzakhia:2008} study AR($1$)-processes with $a_1 \in (0,1)$ and show that $p_N(x)$ decays at least exponentially for a large class of distributions. Bounds on the exponential rate of decay for AR($1$)-processes with Gaussian innovations can be found in \cite{aurzada-baumgarten:2011}. Besides, the decay of the survival probability is known for integrated random walks ($p=2, a_1 = 2, a_2 = -1$): if $\E{}{Y_1} = 0$, $\E{}{Y_1^2} \in (0,\infty)$, it holds that $p_N(x) \asymp N^{-1/4}$ (see \cite{dembo-ding-gao:2012} and the references therein). \\
Taken as a whole, very little is known about the decay of $p_N$ for AR processes except in the few cases mentioned above. As noted in \cite{dembo-ding-gao:2012}, this would be of much interest in view of the frequent appearance of AR-processes and survival probabilities in physical and ecomomic models. Here we investigate the behaviour of the survival probability for such processes under various conditions on the distribution of the innovations. Since an AR($p$)-process $X$ can be written as $X_n = \sum_{k=1}^n c_{n-k} Y_k$ where the $(c_n)$ solve the difference equation $c_n = a_1 c_{n-1} + \dots + a_p c_{n-p}$ with suitable inital conditions, we search criteria for the sequence $(c_n)$ that allow us to characterize the survival probability. Specifically, we are interested in the following question for AR($p$)-processes: when is $p_N$ of polynomial order, when does $p_N$ converge to a positive limit and when is the decay faster than any polynomial?  This classification seems natural if one recalls the results for AR($1$)-processes $X_n = \rho X_{n-1} + Y_n$ where $c_n = \rho^n$ for all $n$. In this case, the behaviour of the survival probability ranges from exponential decay for $\rho < 1$, polynomial decay if $\rho = 1$ and $\E{}{Y_1} = 0$ to convergence to a positive constant if $\rho > 1$. \\
As we will see, the sequence $(c_n)$ often has a much more complex form if $p \geq 2$ so that results for AR($1$)-processes generally cannot be extended directly to higher order processes. We will derive criteria that allow for the classification of the asymptotic behaviour of the $p_N$ as above. Particular emphasis is put on AR($2$)-processes.\\
Let us introduce some notation and conventions: If $f,g: \mathbb{R} \to \mathbb{R}$ are two functions, we write $f \precsim g$ $(x \to \infty)$ if $\limsup_{x \to \infty} f(x)/g(x) < \infty$ and $f \asymp g$ if $f \precsim g$ and $g \precsim f$. Moreover, $f \sim g$ $(x \to \infty)$ if $f(x)/g(x) \to 1$ as $x \to \infty$. If $(X_t)_{t \geq 0}$ is a stochastic process, it will often be convenient to write $X(t)$ instead of $X_t$. If $X$ and $Y$ are random variables, we write $X \stackrel{d}{=} Y$ to denote equality in distribution.\\
The remainder of this article is organized as follows. After presenting the main results for AR($2$) processes below, we start some preliminaries on autoregressive processes in Section~\ref{sec:AR-2}. In Section~\ref{sec:exp-bounds}, we state general conditions ensuring that $p_N$ decays exponentially or at least faster than any polynomial. Special emphasis is put on the case that $(c_n)_{n \geq 0}$ is absolutely summable and AR($2$)-processes. We also prove exponential lower bounds for certain classes of AR-processes. We then determine the region where the survival probability decays polynomially for AR($2$)-processes in Section~\ref{sec:poly}, before briefly treating the case that $p_N$ converges to a positive constant in Section~\ref{sec:pos_lim}.
\subsection{Main results for AR($2$) processes}
Let us illustrate our main result when $X$ is AR(2), i.e. $X_n = a_1 X_{n-1} + a_2 X_{n-2} + Y_n$ with $(Y_n)_{n \geq 1}$ i.i.d. Recall that $X_n = \sum_{k=1}^n c_{n-k} Y_k$ for $n \geq 1$. We decompose $\mathbb{R}^2$ into three disjoint regions $C, E$ and $P$ (see Figure~\ref{fig:R_2_decomp}) defined as follows:
\begin{align*}
   C &:= \set{(a_1,a_2) : a_1 \geq 2, a_1^2 + 4a_2 > 0} \cup \set{(a_1,a_2) : a_1 \in (0,2), a_1 + a_2 > 1 } \\
&\quad \cup \set{(a_1,a_2) : a_1^2 + 4 a_2 = 0, a_1 > 2} \cup \set{(a_1,a_2) : a_1 = 0, a_2 > 1},\\
P &:= \set{(a_1,a_2) : a_1 + a_2 = 1, a_2 \in [-1,1]},\\
 E &:= \mathbb{R}^2 \setminus (C \cup P ).
\end{align*}
\begin{figure}[h!] \label{fig:R_2_decomp}  \centering 
 \includegraphics[scale=0.8]{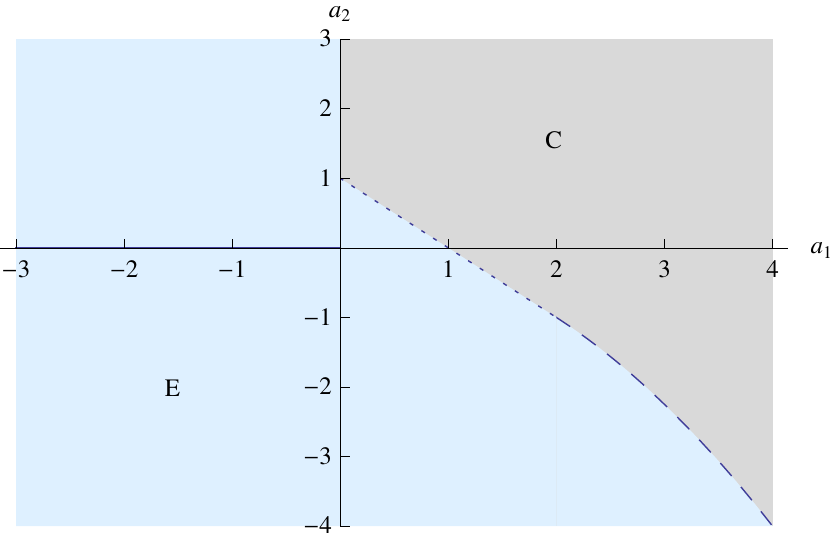}
 \caption{The regions $C$ and $E$. $P$ corresponds to the dotted line. The dashed line is the boundary of $C$ whereas $E$ is open.}
\end{figure}
Depending on the membership of $(a_1,a_2)$ to one of these sets, we can characterize the behaviour of the survival probability under certain conditions on the law of $Y_1$.\\

If $(a_1,a_2) \in P$, the survival probability decays polynomially if $\E{}{Y_1} = 0$ under suitable moment conditions. The choice $a_1 = 2, a_2 = -1$ corresponds to an integrated random walk where $p_N \asymp N^{-1/4}$ if $\E{}{Y_1} = 0$ and $\E{}{Y_1^2} \in (0,\infty)$, see \cite{dembo-ding-gao:2012}. If $a_1 + a_2 = 1$ with $\abs{a_2} < 1$, we will see that $X$ can be seen as a perturbed random walk since $c_n = c + C \epsilon^n$ where $\abs{\epsilon} < 1$. Moreover, $X$ can also be written as an integrated AR(1)-process. The process corresponding to $a_1 = 0,a_2 = 1$ describes two independent random walks such that its survival probability is the square of that of a random walk.
\begin{thm}\label{thm:poly_decay_summary}
   Let $(a_1,a_2) \in P \setminus \set{(2,-1)}$. Assume that $\E{}{Y_1} = 0$ and that $\E{}{e^{ \abs{Y_1}^\alpha} } < \infty$ for some $\alpha > 0$. 
  Then
\[
   p_N = N^{-1/2 + o(1)} \quad (\abs{a_2} < 1), \qquad p_N \asymp N^{-1} \quad (a_2 = 1).
\]
\end{thm}
Next, we also prove that the survival probability decays faster than any polynomial if $(a_1,a_2) \in E$ under certain conditions on the law of $Y_1$. 
\begin{thm}\label{thm:exp_upper_bound_summary}
   Let $(a_1,a_2) \in E$. Assume that $\pr{}{Y_1 > 0} \in (0,1)$, $\E{}{e^{\abs{Y_1}^\alpha}} < \infty$ for some $\alpha > 0$ and that the characteristic function $\varphi$ of $Y_1$ satisfies $\varphi(t) \to 0$ as $\abs{t} \to \infty$. Then $p_N \precsim \exp(-\lambda N / \log N)$ for some $\lambda = \lambda(a_1,a_2) > 0$.
\end{thm}
Actually, we can show that $p_N$ decays at least exponentially on most parts of $E$ under various conditions on $Y_1$. The reason for the rapid decay of the survival probability on $E$ can be explained as follows: either $c_n \to 0$ exponentially fast or $(c_n)$ oscillates and diverges to $\pm \infty$.\\
If $(a_1,a_2) \in C$, we will see that $c_n = \exp(\lambda n (1 + o(1))$ for some $\lambda > 0$. One therefore expects that the process stays below a constant barrier at all times with positive probability. This is confirmed by the following theorem:
\begin{thm}\label{thm:pos_lim_summary}
   Let $(a_1,a_2) \in C$. Assume that $\pr{}{Y_1 < 0} > 0$ and $\pr{}{Y_1 \geq x} \precsim (\log x)^{-\alpha}$ as $x \to \infty$ for some $\alpha > 1$. Then it holds that 
\[
   \pr{}{\sup_{n \geq 1} X_n \leq x} = \lim_{N \to \infty} p_N(x) > 0, \quad x \geq 0.
\]
\end{thm}
Note that the assumption $\E{}{Y_1} = 0$ is essential for the polynomial behaviour of $p_N$ if $(a_1,a_2) \in P$. For instance, if $(S_n)_{n \geq 1}$ is a random walk, it is known that the survival probability can decay polynomially or exponentially if $\E{}{S_1} > 0$ (see \cite{doney:1989}) whereas it converges to a positive constant if $\E{}{S_1} < 0$. In contrast, if $(a_1,a_2) \in E \cup C$, the behaviour of $p_N$ is more stable in the sense that Theorem~\ref{thm:pos_lim_summary} and Theorem~\ref{thm:exp_upper_bound_summary} do not rely on the condition $\E{}{Y_1} = 0$.\\
The best results can be obtained if the innovations are Gaussian, where we can actually prove that $p_N$ admits an exponential upper bound for all $(a_1,a_2) \in E$. Summing up, this leads to the following theorem:
\begin{thm}\label{thm:gaussian-AR_2-summary}
If $Y_1$ is Gaussian with zero mean, the following statements hold:
\begin{enumerate}
    \item $\lim_{N \to \infty} p_N = p_\infty > 0$ if and only if $(a_1,a_2) \in C$,
    \item $p_N \sim c N^{-1}$ iff $(a_1,a_2) = (0,1)$ and  $p_N \asymp N^{-1/4}$ iff $(a_1,a_2) = (2,-1)$,
    \item $p_N = N^{-1/2 + o(1)}$ if and only if $(a_1,a_2) \in P$ and $\abs{a_2} < 1$, and
    \item $p_N \precsim e^{-\lambda N}$ for some $\lambda > 0$ if and only if $(a_1,a_2) \in E$.
\end{enumerate}
\end{thm}
The theorems above are mostly corollaries to more general theorems that are also applicable to AR($p$)-processes if $p \geq 3$ (see e.g.\ Theorem~\ref{thm:summable_seq_1} and \ref{thm:exp_decay_via_char_fct} and Proposition~\ref{prop:Gaussian_double_root} and \ref{prop:AR_1_coeff_greater_one} below).  We will indicate possible extensions troughout the article. The main advantage of focussing on AR(2)-processes consists of the fact that we have an explicit solution of the difference equation for the sequence $(c_n)_{n \geq 0}$. For instance, this allows us to explicitly describe the parameters $(a_1,a_2)$ such that $c_n \to 0$.\\
Even for AR($2$)-processes, one is forced to distinguish a variety of cases that require different treatment. It is clear that this becomes much more complicated for processes of higher order. Finally, let us mention that the class of AR($p$)-processes contains $p$-times integrated centered random walks  $S^{(p)}$ as a special case (i.e. $S^{(1)}$ is a centered random walk, and $S^{(p)}_n = \sum_{k=1}^n S^{(p-1)}_k$). Here, the behaviour of the survival probability is not known for $p \geq 3$.

\section{Autoregressive processes}\label{sec:AR-2}
We begin by recalling a few facts about AR($p$)-processes. For fixed $p \geq 1$, define $X_{n} = \sum_{k=1}^p a_k X_{n-k} + Y_{n}$, $n \geq 0$ with the convention that $X_n = 0$ for $n \leq 0$ where $(Y_n)_{n \geq 1}$ is a sequence of i.i.d.\ random variables. One verifies that $X_n = \sum_{k=1}^n c_{n-k} Y_k$ where
\[
   c_n = 0, \quad n < 0, \qquad c_0 = 1, \qquad c_n = \sum_{k=1}^p a_k c_{n-k}, \quad n \geq 1.
\]
In other words, $(c_n)_{n \geq 0}$ solves the linear difference equation 
\begin{align*}
   c_n = a_1 c_{n-1} + \dots a_p c_{n-p}, n \geq p,
\end{align*}
with initial conditions 
\[
   c_0 = 1, \quad c_1 = a_1 c_0, \quad c_2 = a_1 c_1 + a_2 c_0, \quad \dots, \quad c_{p-1} = a_1 c_{p-2} + \dots + a_{p-1} c_0.
\]
Solving this equation amounts to finding the roots $s_1,\dots,s_p \in \mathbb{C}$ of the characteristic polynomial $f_p(\cdot)$, given by $f_p(x) := x^p - \sum_{k=1}^p a_k x^{p-k}, x \in \mathbb{R} \label{def:char_pol}$. If $p=2$, the roots $s_1, s_2$ of $f_2(\lambda) = \lambda^2 - a_1 \lambda - a_2$ are given by
\begin{equation}\label{eq:sol_diff_eq_2}
   s_1 := (a_1 + h)/2, \quad s_2 := (a_1-h)/2, \quad h := \sqrt{a_1^2 + 4a_2} \in \mathbb{C}.
\end{equation}
Taking into account the inital conditions $c_0=1,c_1= a_1$, one can show that 
\begin{equation}\label{eq:sol_diff_eq}
   c_n = \begin{cases}
            h^{-1} \, \left( s_1^{n+1} - s_2^{n+1} \right) , &\quad n \geq 0, \quad a_1^2 + 4 a_2 \neq 0,\\
	    \left( a_1/2 \right)^n ( n  + 1), &\quad n \geq 0, \quad a_1^2 + 4 a_2 = 0.
         \end{cases}
\end{equation}
If $a_1^2 + 4a_2 < 0$, writing $s_1 = r e^{i\varphi}$ and $s_2= r e^{-i\varphi}$ in polar form, elementary manipulations show that the solution is given by
\begin{equation}\label{eq:sol_diff_eq_sin_cos}
  c_n = \abs{a_2}^{n/2} \left( \cos(n\varphi) + \frac{a_1}{\tilde{h}} \, \sin(n\varphi) \right),
\end{equation}
where
\[
   \tilde{h} = \sqrt{-(a_1^2 + 4a_2)}, \quad 
\varphi = \begin{cases} \arctan(\tilde{h}/a_1) \in (0,\pi/2), &\quad a_1 >0,\\
             \pi/2, &\quad a_1 = 0,\\
	     \pi + \arctan(\tilde{h}/a_1) \in (\pi/2,\pi), &\quad a_1 < 0.
          \end{cases}
\]
\begin{remark}\label{rem:c_n_to_zero}
   It holds that $c_n \to 0$ if and only if $\max \set{\abs{s_1},\abs{s_2}} < 1$ which is equivalent to the conditions
\[
   a_1 + a_2 < 1, \quad a_2 < 1 + a_1, \quad a_2 > -1,
\]
see Theorem~2.37 of \cite{elaydi:1999}.
\begin{figure}[h!]\label{fig:c_n_to_zero}  \centering 
 \includegraphics[scale=0.5]{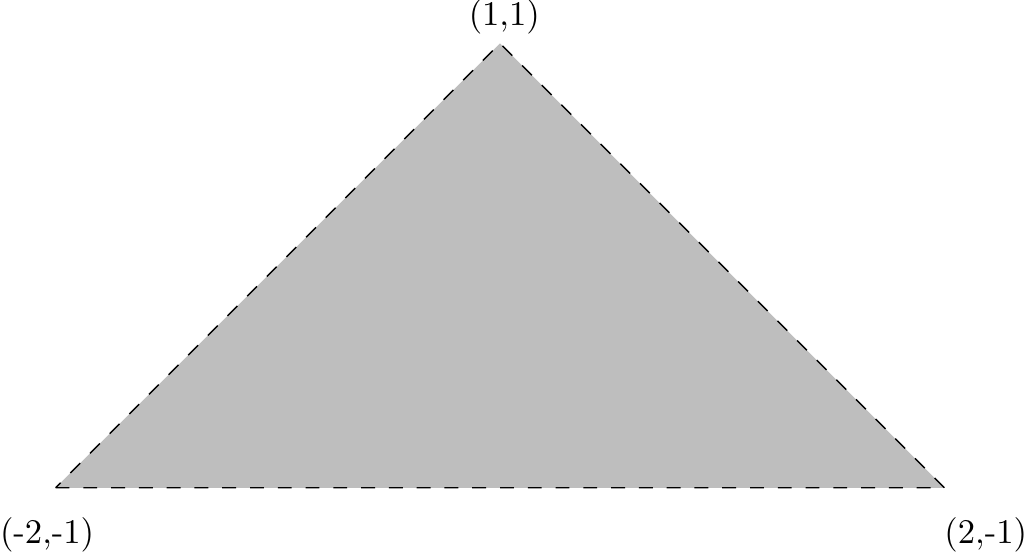} 
 \caption{The region of parameters $(a_1,a_2)$ where $c_n \to 0$} 
\end{figure}
\end{remark}
\begin{remark}
Note that the convention that $X_n = 0$ for $n < 0$ is not standard to define autoregressive processes. It is often customary to define AR($p$)-processes as follows, see e.g.\ Chapter~3 in \cite{brockwell-davis:1987}: If $(Y_n)_{n \in \mathbb{Z}}$ is a sequence of i.i.d. random variables, $X = (X_n)_{n \in \mathbb{Z}}$ is AR($p$) if
\[
   X_n = a_1 X_{n-1} + \dots + a_p X_{n-p} + Y_n, \quad n \in \mathbb{Z}.
\]
Moreover, $X$ is called causal if there exists a deterministic sequence $(c_n)_{n \ge 0}$ with $\sum \abs{c_n} < \infty$ such that $X_n = \sum_{k=0}^\infty c_n Y_{n-k}$.\\
By Theorem~3.1.1 of \cite{brockwell-davis:1987}, $X$ is causal if and only if the polynomial $p(z) = 1 - a_1 z - \dots - a_p z^p$ has no zeros in $\set{ z \in \mathbb{C} : \abs{z} \le 1}$. In that case, the coefficients $c_n$ are determined by the following relation $\sum_{k=0}^\infty c_k z^k = 1/p(z)$ for $\abs{z} \le 1$. Equating the coefficients of $z^k$, one easily verifies (or see Section~3.3 in \cite{brockwell-davis:1987}) that the sequence $(c_n)_{n \ge 0}$ satisfies the same recursion equation with the same initial conditions as above. Hence, it $X$ is a causal AR($p$)-process, we can decompose it for $n \ge 1$ in the following way:
\begin{align*}
   X_n = \sum_{k=0}^{n-1} c_k Y_{n-k} + \sum_{k=n}^\infty c_k Y_{n-k} = \sum_{k=1}^n c_{n-k} Y_k + \sum_{k=0}^\infty c_{n+k} Y_{-k} = X^{(1)}_n + X^{(2)}_n.
\end{align*}
Note that $X^{(1)}$ and $X^{(2)}$ are independent and that $X^{(1)}$ is an AR($p$)-process in the sense of this article. The term $X^{(2)}$ can be seen as a small perturbation since $\E{}{\abs{X^{(2)}_n}} \le \E{}{\abs{Y_1}} \, \sum_{k=n}^\infty c_k \to 0$.\\
Moreover, the fact that $c_n \to 0$ allows us to apply Theorem~\ref{thm:summable_seq_2} below if $X$ AR($p$) in the sense of Brockwell and Davis. By using the alternative definition above, we can also define autoregressive processes when $c_n$ does not go to zero and we get a much larger class of processes including, for example, random walks.
\end{remark}

We will use different methods to prove certain statements about the survival probability depending on the parameters $(a_1,a_2)$. To this end, set
\begin{align*}
   E_1 &:= \set{ (a_1,a_2) : a_1 < 0, a_2 > 0, a_2 > 1 + a_1}, \quad E_2 := (-\infty,0]^2,\\
  E_3 &:= \set{ (a_1,a_2) : a_1 >0 ,a_1^2 + 4a_2 < 0}. 
\end{align*}
 Figure~\ref{fig:2} will be helpful to visualize the regions that will be considered separately below.\\
\begin{figure}[h!] \label{fig:2}  \centering 
 \includegraphics[scale=1]{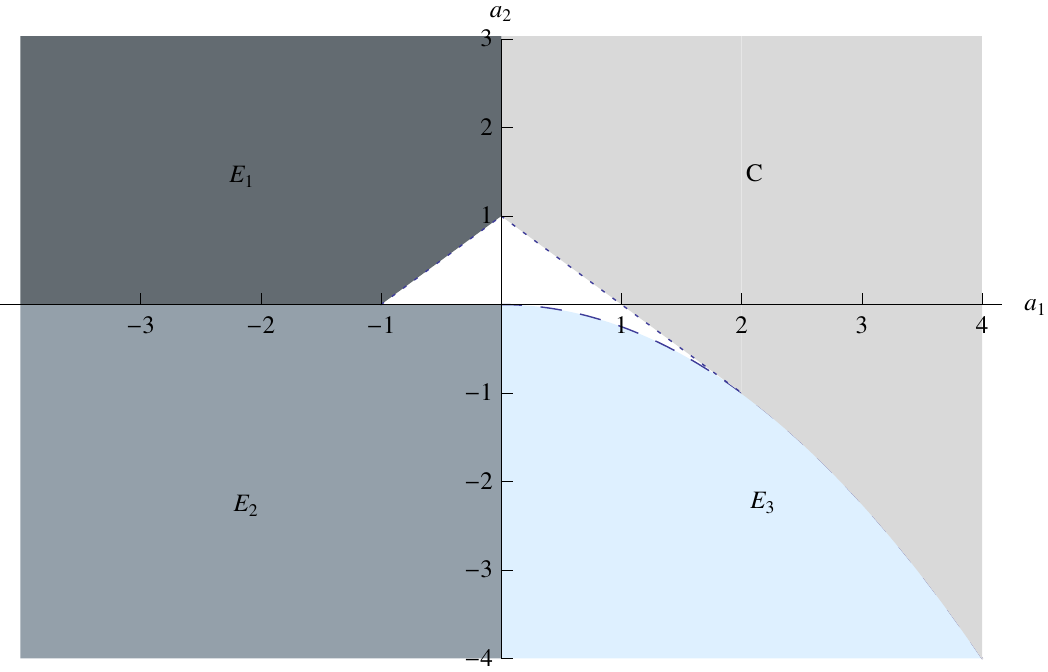} 
 \caption{The regions $E_1,E_2,E_3$ and $C$} 
\end{figure}
Let us also comment briefly on the dependence of the survival probability on the barrier $x$ for AR($p$)-processes. In principle, the behaviour of the survival probability can vary significantly for different barriers. An extreme example is an AR(1)-process $Z_n = \rho Z_{n-1} + Y_n$ where $\rho \in (0,1)$ with $\pr{}{Y_1 = 1} = \pr{}{Y_1 = -1} = 1/2$. It is known that $p_N \precsim \exp(-\lambda N)$ for some $\lambda > 0$ (see Theorem~\ref{thm:AR_1_nov-kord} below), whereas $p_N(x) = 1$ for all $x \geq 1/(1-\rho)$ since $\abs{X_n} = \abs{\sum_{k=1}^n \rho^{n-k} Y_k} \leq \sum_{k=0}^\infty \rho^k = 1/(1-\rho)$.\\
On the other hand, if $c_n \geq \delta > 0$ for all $n \geq 0$ and $\pr{}{Y_1 \leq -\epsilon} > 0$ for some $\epsilon > 0$, one can show that $p_N(x) \asymp p_N$ as $N \to \infty$ for all $x \geq 0$. Indeed, note that if $Y_1 \leq -\epsilon$, it follows that $X_n = c_{n-1} Y_1 + \sum_{k=2}^n c_{n-k} Y_k \leq -\epsilon \delta + \sum_{k=2}^n c_{n-k} Y_k$, so that
\begin{align*}
   p_N = \pr{}{\sup_{n=1,\dots,N} X_n \leq 0} &\geq \pr{}{Y_1 \leq -\epsilon} \pr{}{\sup_{n=2,\dots,N} \sum_{k=2}^n c_{n-k} Y_k \leq \delta \epsilon} \\
&\geq \pr{}{Y_1 \leq -\epsilon} p_N(\delta \epsilon).
\end{align*}
Iteration shows that $p_N \geq \pr{}{Y_1 \leq -\epsilon}^L p_N(L \delta \epsilon)$ for $L = 1,\dots,N$. Hence, if $x \geq 0$, take $L$ with $L \delta \epsilon \geq x$ to get that $\pr{}{Y_1 \leq -\epsilon}^L p_N(x) \leq p_N \leq p_N(x)$ for all $N$ large enough.
\section{Exponential bounds}\label{sec:exp-bounds}
\subsection{Exponential upper bounds}
Let us begin with a trivial observation: If $a_1 \leq 0, \dots, a_p \leq 0$, we have that
\[
   \pr{}{\sup_{n=1,\dots,N} X_n \leq 0} \leq \pr{}{Y_1 \leq 0}^N,
\]
since $X_1 \leq 0, \dots, X_n \leq 0$ implies that $Y_k \leq -a_1 X_{k-1} - \dots - a_p X_{k-p} \leq 0$ for all $k=1,\dots,n$. If $p=2$, this shows that $p_N$ decays exponentially on $E_2$, see Figure~\ref{fig:2}.\\
As we will see in the sequel, exponential decay of $p_N$ occurs for two differnt reasons: first, if $c_n \to 0$ and second, if $(c_n)_{n \geq 0}$ oscillates and diverges exponentially fast. \\
Let us first consider the case that $c_n$ goes to zero. Recall that for AR($1$)-processes $(Z_n)_{n \geq 1}$ with $Z_n = \rho Z_{n-1} + Y_n$ for $\rho \in (0,1)$, $c_n = \rho^n \to 0$ and $p_N$ decays exponentially under mild assumptions on the distribution of $Y_1$ by  
Theorem~1 of \cite{novikov-kordzakhia:2008}:
\begin{thm}\label{thm:AR_1_nov-kord}
   Let $0 < \rho < 1$, $x \geq 0$ and assume that $\E{}{(Y_1^-)^\delta} < \infty$ for some $\delta \in (0,1)$ and $\pr{}{Y_1 > x(1-\rho)} > 0$. Then $\E{}{\exp(\alpha \tau_x)} < \infty$ for some $\alpha > 0$.
\end{thm}
We now state a similar weaker result that provides a simple criterion for AR($p$)-processes to ensure that $p_N$ decays faster to zero than any polynomial.
\begin{thm}\label{thm:summable_seq_1}
    Let $(c_k)_{k \geq 0}$ denote a sequence with $c_0 = 1$ and $A := \sum_{k=0}^\infty \abs{c_k} < \infty$ such that $\sum_{k=q}^\infty \abs{c_k} \leq C e^{-\lambda q}$ for every $q \geq 1$ ($C,\lambda > 0$ constants). Assume that there is $\delta > 0$ with $\pr{}{Y_1 < -\delta} > 0$ and $\pr{}{Y_1 > \delta} > 0$ and that $\E{}{Y_1^2} < \infty$. Let $X_n = \sum_{k=1}^n c_{n-k} Y_k$. 
Then for $x \in [0,\delta A)$, there is $c(x) > 0$ such that
 \[
    p_N(x) \precsim \exp \left( -c(x) \, \sqrt{N} \right), \quad N \to \infty.
 \]
Moreover, if $\E{}{\exp(\abs{Y_1}^\alpha)} < \infty$ $(\alpha > 0)$ and $x \in [0, \delta A)$, there is $c(x) > 0$ such that
 \[
    p_N(x) \precsim \exp \left( -c(x) \, N / \log N \right), \quad N \to \infty.
 \]
 \end{thm}
 \begin{proof}
   For $q \geq 1$, define $Z_{q,n} = \sum_{k=n-q}^n c_{n-k} Y_k$ for $n \geq q+1$. Note that $Z_{q,n}$ is measurable w.r.t.\ $\sigma(Y_{n-q},\dots,Y_n)$ which implies that $(Z_{q,n(q+1)})_{n \geq 1}$ defines a sequence of i.i.d.\ random variables with $Z_{q,q+1} = X_{q+1}$. We will show that $Z_{q,n}$ is a good approximation of $X_n$ if $q$ is large. We then obtain an estimate on $p_N(x)$ by computing the survival probability of the independent r.v.\ $(Z_{q,(q+1)n})_{n \ge 1}$.\\
First, observe that
\begin{align}
 \pr{}{ \sup_{n=q+2,\dots,N} \abs{X_n - Z_{q,n}} > u} &\le \sum_{n=q+2}^N \pr{}{ \abs{\sum_{k=1}^{n-q-1} c_{n-k} Y_k} > u} \notag \\
&= \sum_{n=q+2}^N \pr{}{ \abs{\sum_{k=q+1}^{n-1} c_k Y_k} > u} =: h_N(u).
\end{align}
In the first equality, we have used that the $Y_k$ are i.i.d., and therefore exchangeable.\\
Hence, 
\begin{align}
   \pr{}{\sup_{n=1,\dots,N} X_n \leq x} &\leq \pr{}{\sup_{n=q+2,\dots,N} Z_{q,n} \leq x + \epsilon} + h_N(\epsilon) \notag \\
&\leq \pr{}{\sup_{n=1,\dots,\lfloor N / (q+1) \rfloor } Z_{q,n(q+1)} \leq x + \epsilon} + h_N(\epsilon) \notag \\
&= \pr{}{ Z_{1,q+1} \leq x + \epsilon }^{\lfloor N/(q+1) \rfloor} + h_N(\epsilon), \label{eq:proof_exp_upper_bound}
\end{align}
where we have used the fact that $(Z_{q,n(q+1)})_{n \ge 1}$ is an i.i.d.\ sequence.
Since the $(Y_n)$ are i.i.d.\ (and therefore exchangeable), we get for $y \in \mathbb{R}$ that
\begin{equation}\label{eq:proof_exp_decay_conv_distr}
   \pr{}{Z_{1,q+1} \leq y} = \pr{}{\sum_{k=0}^{q} c_k Y_{k+1} \leq y} \to \pr{}{\sum_{k=0}^\infty c_k Y_{k+1} \leq y}, \quad q \to \infty,
\end{equation}
since the series $\sum_{k=0}^\infty c_k Y_{k+1} =: Z$ converges a.s.\ by Kolmogorov's Three Series Theorem since $\E{}{Y_1^2} < \infty$. Next, $\pr{}{Z \leq y} < 1$ for every $0 \le y < \delta A$ by Theorem~3.7.5 of \cite{lukacs:1970}. Then for $0 \le y < \delta A$, by \eqref{eq:proof_exp_decay_conv_distr}, there is $\rho = \rho(y)  < 1$ such that $\pr{}{Z_{1,q + 1} \leq y} \leq \rho$ for all $q$ sufficiently large.\\
Moreover, using first Chebychev's inequality and our assumptions on the sequence $(c_n)_{n \ge 0}$, we obtain that
\begin{align}
   h_N(u) &\le u^{-1} \sum_{n=q+2}^N \sum_{k=q+1}^{n-1} \abs{c_k} \E{}{\abs{Y_1}} \le u^{-1} \E{}{\abs{Y_1}} \sum_{n=q+2}^N  \sum_{k=q+1}^\infty \abs{c_k} \notag \\
&\le C N u^{-1} \E{}{\abs{Y_1}} e^{-\lambda q} = C_1 N u^{-1} e^{-\lambda q}. \label{eq:bound_h_N}
\end{align}
Let $q = q_N := \lfloor \beta \sqrt{N} \rfloor$, $\beta > 0$. If $u>0$ is such that $x+u < A$, we deduce from \eqref{eq:proof_exp_upper_bound} and \eqref{eq:bound_h_N} that
\begin{align*}
   p_N(x) \le \rho^{\sqrt{N} / \beta} + C_1 N u^{-1} e^{-\lambda \beta \sqrt{N} + \lambda}.
\end{align*} 
By choosing $\beta$ sufficiently large, the theorem follows under the assumption $\E{}{Y_1^2} < \infty$. \\
If $\E{}{\exp(\abs{Y_1}^\alpha)} < \infty$, the estimate on $h_N$ can be improved as follows: 
\[
   \sup_{n=q+1,\dots,N} \abs{ \sum_{k=q+1}^{n} c_k Y_k } \le \sup_{l=q+1,\dots,N} \abs{Y_k} \sum_{k=q+1}^{\infty} \abs{c_k} \le C e^{-\lambda q} \sup_{l=q+1,\dots,N} \abs{Y_k}. 
\]
Hence,
\begin{align*}
   h_N(u) &\le \sum_{n=q+1}^N \pr{}{ \sup_{k=q+1,\dots,N} \abs{Y_k} > e^{\lambda q} u /C } \le N^2 \pr{}{ \abs{Y_1} > e^{\lambda q} u / C} \\
&\le N^2 \exp \left( -e^{\alpha \lambda q} (u / C)^\alpha  \right) \E{}{\exp(\abs{Y_1}^\alpha)}.
\end{align*}
In particular, with $q = q_N = \lfloor \kappa \log N \rfloor$, if $\kappa$ is large enough, this implies together with \eqref{eq:proof_exp_upper_bound} that, for some $c(x) > 0$,
\[
   p_N(x) \precsim N^2 e^{-N^2} + \rho^{\lfloor N/(q_N + 1) \rfloor} \precsim \exp(-c(x) N / \log N), \quad N \to \infty.
\]
\end{proof}
The proof of Theorem~\ref{thm:summable_seq_1} reveals that fast decay of $p_N$ be explained intuitively as follows: if we write $X_n = \sum_{k=1}^{n-q-1} c_{n-k} Y_k + \sum_{k=n-q}^n c_{n-k} Y_k$, the first summand is typically small if $q$ is large and $c_n \to 0$. Hence, 
\[
   \pr{}{\sup_{n=1,\dots,N} X_n \le )} \approx \pr{}{\sup_{n=q+1,\dots,N} \sum_{k=n-q}^n c_{n-k} Y_k \le 0 } \approx \pr{}{ \sum_{k=1}^{q+1} c_{n-k} Y_k \le 0}^{N/q}.
\]
\begin{remark}
   \label{rem:c_n_summable_bounded_innov}
If $(c_k)_{k \geq 0}$ denote a sequence with $c_0 = 1$ and $\sum_{k=0}^\infty \abs{c_k} < \infty$ and $\abs{Y_1} \leq M$ a.s.\ for some $M < \infty$, one can prove in an analogous way that even $p_N \precsim \exp(-c N)$ for some $c > 0$ since $h_N(u)$ in the proof of Theorem~\ref{thm:summable_seq_1} vanishes for $q$ large enough..
\end{remark}
\begin{remark}
   \label{rem:pos_c_n_simplification}
As it was already remarked by \cite{novikov-kordzakhia:2008}, if $(c_k)_{k \geq 0}$ denotes a sequence of positive numbers, one has that
\[
   X_n = \sum_{k=1}^n c_{n-k} Y_k \ge \sum_{k=1}^n c_{n-k} Y_k \indic{Y_k \le M} = \sum_{k=1}^n c_{n-k} \tilde{Y}_k =: \tilde{X}_n,
\]
such that $\pr{}{X_n \le x, \forall n \le N} \le \pr{}{\tilde{X}_n \le x, \forall n \le N}$. Hence, if the $c_n$ are positive, one can assume without loss of generality that the innovations are bounded from above in order to establish an upper bound on the survival probability.
\end{remark}
For AR($2$)-processes, Theorem~\ref{thm:summable_seq_1} is applicable if $a_1 + a_2 < 1$, $a_2 < a_1 + 1$ and $a_2 > -1$, cf.\ Remark \ref{rem:c_n_to_zero} and Figure~\ref{fig:c_n_to_zero}. Moreover, the preceding theorem can be generalized easily to cover more general processes (such as autoregressive moving average models ARMA(p,q) and moving average processes of infinte order MA($\infty$), see Section~3 in \cite{brockwell-davis:1987}):
\begin{thm}\label{thm:summable_seq_2}
   Let $(c_k)_{k \in \mathbb{Z}}$ denote a sequence with $c_ 0 = 1$, $A := \sum_{k = - \infty}^\infty \abs{c_k} < \infty$ and $\sum_{\abs{k} \geq q} \abs{c_k} \leq C e^{-\lambda q}$ for all $q \geq 1$ and some $\lambda > 0$. Let $(Y_k)_{k \in \mathbb{Z}}$ be a sequence of i.i.d.\ random variables such that $\E{}{Y_1^2} < \infty$ and $\pr{}{Y_1 > \delta} > 0$ and $\pr{}{Y_1 < -\delta} > 0$ for some $\delta > 0$. Let $X_n := \sum_{k=-\infty}^\infty c_{n-k} Y_k$ for $n \in \mathbb{Z}$. If $x \in [0,\delta A)$, it holds for some $c(x) > 0$ that
\[
   \pr{}{\sup_{\abs{n} \leq N} X_n \leq x} \precsim \exp( -c(x) \sqrt{N} ), \quad N \to \infty.
\]
Moreover, if $\E{}{\exp(\abs{Y_1}^\alpha)} < \infty$ $(\alpha > 0)$ and $x \in [0,\delta A)$, there is $c(x) > 0$ such that
\[
   \pr{}{\sup_{\abs{n} \leq N} X_n \leq x} \precsim \exp(-c(x) N/\log N), \quad N \to \infty.
\]
\end{thm}
\begin{proof}
  It is well known that $X_n$ is well defined for every $n \in \mathbb{Z}$ under the given assumtions on the sequence $(c_n)$. The proof is then very similar to that of Theorem~\ref{thm:summable_seq_1}. We define $Z_{q,n} := \sum_{k=n-q}^{n+q} c_{n-k} Y_k$. Note that $(Z_{q,n(2q +1)})_{n \in \mathbb{Z}}$ forms a sequence of i.i.d.\ random variables with $Z_{q,0} = \sum_{k=-q}^q c_{k} Y_k$. The remainder of the proof is along the same lines of the proof of Theorem~\ref{thm:summable_seq_1}.
\end{proof}
In certain special cases, we can improve Theorem~\ref{thm:summable_seq_1}. Namely, if $(c_n)$ is a sequence of positive numbers and $c_n = \rho^n(1 + o(1))$ where $\rho \in (0,1)$, it follows from Theorem~\ref{thm:AR_1_nov-kord} that $p_N$ goes to zero exponentially fast under mild assumptions on $Y_1$:
\begin{prop}\label{prop:gen_of_nov-kord}
   Let $(c_n)_{n \geq 0}$ be a sequence such that $\alpha C \rho^n \leq c_n \leq C \rho^n$ for all $n \geq 0$ where $\rho \in (0,1)$, $0< \alpha < 1$, $C >0$. Assume that $\E{}{(Y_1^-)^\delta}< \infty$ for some $\delta \in (0,1)$. Let $x \geq 0$ be such that $\pr{}{ Y_1 \geq x ( 1 - \rho)/(\alpha C) } > 0$. Let $X_n := \sum_{k=1}^n c_{n-k} Y_k$. Then there is some $\lambda = \lambda(x) > 0$ such that $p_N(x) \precsim \exp(-\lambda N)$. 
\end{prop}
\begin{proof}
Define the i.i.d.\ random variables $\tilde{Y}_k := Y_k \indic{Y_k < 0} + \alpha Y_k \indic{Y_k > 0}$, $k \geq 1$. Since $c_k \geq 0$ for all $k$, we obtain that
\begin{align*}
   X_n = \sum_{k=1}^n c_{n-k} Y_k \geq \sum_{k=1}^n C \rho^{n-k} Y_k \indic{Y_k < 0} + \sum_{k=1}^n \alpha C \rho^{n-k} Y_k \indic{Y_k >0} = C \sum_{k=1}^n \rho^{n-k} \tilde{Y}_k =: C Z_n,
\end{align*}
where $Z_n := \rho Z_{n-1} + \tilde{Y}_n$. In particular, we conclude that 
\[
   \pr{}{\sup_{n=1,\dots,N} X_n \leq x} \leq \pr{}{\sup_{n=1,\dots,N} Z_n \leq x/C}.
\]
Now $\pr{}{\tilde{Y}_1 > x(1 - \rho)/C} = \pr{}{ Y_1 \geq x ( 1 - \rho)/(\alpha C) } > 0$ by the choice of $x$. Hence, the result follows from Theorem~1 of \cite{novikov-kordzakhia:2008} (Theorem~\ref{thm:AR_1_nov-kord} above).
\end{proof}
The preceding proposition yields the following corollary for AR($2$)-processes:
\begin{cor}
   Let $a_1 \in (0,2),a_2 < 0$ with $a_1 + a_2 < 1$ and $a_1^2 + 4a_2 > 0$. Assume that $\E{}{(Y_1^-)^\delta} < \infty$ for some $\delta \in (0,1)$ and $\pr{}{Y_1 \geq y} > 0$ for every $y$. For every $x \geq 0$, there is $\lambda = \lambda(x) > 0$ such that $p_N(x) \precsim \exp(-\lambda N)$.
\end{cor}
\begin{proof}
   It is not hard to check that $0 < s_2 < s_1 < 1$. Hence, $c_n = s_1^n (s_1 - s_2 (s_2/s_1)^n)/h$ and $h^{-1}(s_1 - s_2) s_1^n \leq c_n \leq h^{-1} s_1^{n+1}$ for all $n$. The result follows from Proposition~\ref{prop:gen_of_nov-kord}.
\end{proof}
If $\abs{Y_1} \leq M$ a.s., the preceding corollary is not applicable. However, we already know that $p_N \precsim e^{-c N}$ for some $c > 0$ in that case, see Remark~\ref{rem:c_n_summable_bounded_innov}.\\

Let us now establish exponential upper bounds for $p_N$ for certain distributions if the sequence $(c_n)$ oscillates and diverges exponentially. The proof relies on the following proposition.
\begin{prop}\label{prop:small-dev-of-series-char-fct}
Let $\rho \in (-1,1)$ ($\rho \neq 0$) and set $Z := \sum_{n=1}^\infty \rho^n Y_n$.  Moreover, suppose that $\E{}{\abs{Y_1}^a} < \infty$ for some $a > 0$. Let $\varphi$ denote the characteristic function of $Y_1$ and assume that there are $\delta \in(0,\abs{\rho})$ and $t_0 > 0$ such that $\abs{\varphi(t)} \le \delta$ forall $\abs{t} \ge t_0$. It follows that $\pr{}{\abs{Z} \le \epsilon} \precsim \epsilon$ as $\epsilon \downarrow 0$.
\end{prop}
\begin{proof}
$Z$ is well-defined and its characterisitic function $\tilde{\varphi}$ is given by $\tilde{\varphi}(t) = \prod_{n=1}^\infty \varphi(\rho^n t)$, see e.g.\ Section~3.7 of \cite{lukacs:1970}. Let us show that $\tilde{\varphi}$ is absolutely integrable. If this holds, by Theorem 3.2.2 of \cite{lukacs:1970}, $Z$ admits a continuous density $g$ which is given by 
\[
   g(x) := \frac{1}{2 \pi} \int_{-\infty}^\infty e^{-ixt} \tilde{\varphi}(t) \, dt, \quad x \in \mathbb{R}.
\]
In particular, $g$ is bounded implying that $\pr{}{\abs{Z} \leq \epsilon} \leq C \, \epsilon$ for any $\epsilon \geq 0$. \\
To prove the integrability of $\tilde{\varphi}$, let $\delta$ and $t_0$ be as in the statement of the proposition and note that
\[
	\abs{\tilde{\varphi}(t)} = \prod_{n=1}^{\infty} \abs{\varphi(\rho^n t)} \le \delta^{N(t)},
\]
where $N(t) = \#\set{ n \ge 1: \abs{\rho^n t} \ge t_0 }$. One verifies that $N(t) = \lfloor ( \log(t) - \log(t_0) ) / \log(1/\abs{\rho}) \rfloor$ so that
\begin{align*}
	\abs{\tilde{\varphi}(t)} \le \exp \left( \log \delta \, \left( \frac{\log \abs{t} - \log(t_0)}{ \log(1/\abs{\rho})} - 1 \right) \right) = C \abs{t}^{-\alpha},
\end{align*}
where $C$ depends on $t_0, \rho$ and $\delta$ only and $\alpha := \log(1/\delta)/\log(1/\abs{\rho}) > 1$. This shows that $\abs{\tilde{\varphi}(t)}$ is integrable over $\mathbb{R}$.
\end{proof}
\begin{remark}
Recall that if $X$ has an absolutely continuous distribution, it holds that $\lim_{\abs{t} \to \infty} \E{}{e^{itX}} = 0$, see e.g.\ Section 2.2 in \cite{lukacs:1970}. However, if the distribution of $X$ is purely discrete, $\limsup_{\abs{t} \to \infty} \abs{\E{}{e^{itX}}} = 1$ and in general, it is a very challenging problem to find conditions such that the random series $\sum_{k=1}^\infty \rho^n Y_n$ has a density. This question has attracted a lot of attention for so-called infinite Bernoulli convolutions. We refer to the survey of \cite{peres-solomyak:2000}.
\end{remark}
We can now prove the following theorem.
\begin{thm}\label{thm:exp_decay_via_char_fct}
Let $X_n := \sum_{k=1}^n c_{n-k} Y_k$ where $c_n = d \rho^n + \beta_n r^n$ where $d \neq 0$, $\rho < - 1$ and $\abs{\rho} > \abs{r}$ and $\abs{\beta_n}e^{-\lambda n} \to 0$ as $n \to \infty$ for every $\lambda > 0$. Assume $\E{}{\abs{Y_1}^a} < \infty$ for some $a > 0$. Moreover, suppose that the characteristic function $\varphi$ of $Y_1$ satisfies the inequality $\abs{\varphi(t)} \le \delta < \abs{\rho}$ for all $\abs{t}$ large enough. Then there is a constant $C > 0$ such that for every $x \geq 0$, it holds that
\[
  \liminf_{N \to \infty} - N^{-1} \log \pr{}{\sup_{n=1,\dots,N} X_n \leq x} \geq C.
\]
If $\E{}{\exp(\abs{Y_1}^\alpha)} < \infty$ for some $\alpha > 0$, then 
\[
 C \geq 
 \begin{cases}
 \log \abs{\rho/r}, &\abs{r} > 1, \\
 \log \abs{\rho}, &\quad \text{else}. 
 \end{cases}
\]
\end{thm}
\begin{proof}
Assume w.l.o.g.\ that $d=1$ (write $X_n = \sum_{k=1}^n (c_{n-k} / d) (d Y_k)$). Let $\hat{\beta}_n := \sup \set{\abs{\beta_0},\dots,\abs{\beta_n}}$ and $E_N := \set{\abs{Y_1} \leq f_N, \dots, \abs{Y_N} \leq f_N}$ where $1 \leq f_N \to \infty$ is to be specified later. On $E_N$, it holds for $n=1,\dots,N$ that
\begin{align*}
   X_n &= \sum_{k=1}^n c_{n-k} Y_k = \sum_{k=1}^n \rho^{n-k} Y_k + \sum_{k=1}^n \beta_{n-k} r^{n-k} Y_k \\
&\geq \sum_{k=1}^n \rho^{n-k} Y_k - \hat{\beta}_n f_N \sum_{k=1}^n \abs{r}^{n-k} \geq \sum_{k=1}^n \rho^{n-k} Y_k - \hat{\beta}_N f_N \sum_{k=0}^N \abs{r}^{k}.
\end{align*}
\textbf{Case 1:} Consider first the case that $\beta_n \neq 0$ for some $n$. Let $R_N := \sum_{k=0}^N \abs{r}^{k}$. Then
\begin{equation}\label{eq:reduction_to_AR_1}
   p_N(x) \leq \pr{}{E_N^c}+ \pr{}{\sup_{n=1,\dots,N} \sum_{k=1}^n \rho^{n-k} Y_k \leq x + \hat{\beta}_N f_N R_N, E_N}.
\end{equation}
Note that $Z_n := \sum_{k=1}^n \rho^{n-k} Y_k$ is an AR($1$)-process satisfying $Z_n = \rho Z_{n-1} + Y_n$. Let us begin with the following useful observation: if $Z_{N-1} \leq z$ and $Z_N \leq z$ for some large $z > 0$, we have with high probability that $\abs{Z_{N-1}} \leq z$. This will allow us to reduce the estimation of $p_N(x)$ to controlling $\pr{}{\abs{Z_N} \leq z_N}$ where $z_N \to \infty$ as $N \to \infty$. To be precise, note that
\begin{align}
   \set{Z_{N-1} \leq z, Z_N \leq z} &\subseteq \set{ \abs{Z_{N-1}} \leq z} \cup \set{Z_{N-1} < -z, Z_N \leq z} \notag \\
&\subseteq  \set{ \abs{Z_{N-1}} \leq z} \cup \set{Y_N \leq (1- \abs{\rho})z }. \label{eq:proof_char_fct_1}
\end{align}
For the last inclusion, we have used that the event $\set{Z_{N-1} < -z, Z_N \leq z}$ implies that $z \geq Z_N = \rho Z_{N-1} + Y_N \geq - \rho z + Y_N$. Hence, combining this with \eqref{eq:reduction_to_AR_1}, we obtain that
\begin{align}
    p_N(x) &\leq \pr{}{E_N^c}+ \pr{}{Z_{N-1} \leq x + \hat{\beta}_N f_N R_N, Z_N \leq x + \hat{\beta}_N f_N R_N} \notag \\
&\leq \pr{}{E_N^c} + \pr{}{\abs{Z_{N-1}} \leq  x + \hat{\beta}_N f_N R_N} + \pr{}{Y_N \leq (1-\abs{\rho})  (x + \hat{\beta}_N f_N R_N)}.
\label{eq:proof_char_fct}
\end{align}
It remains to estimate the three probabilities above. Clearly,
\begin{equation*}
   \pr{}{E_N^c} = \pr{}{\bigcup_{n=1}^N \set{\abs{Y_N} > f_N}} \leq N \pr{}{\abs{Y_1} > f_N}.
\end{equation*}
Next, since $\abs{\rho} > 1$ and $\hat{\beta}_N \geq \beta > 0$ for some $\beta > 0$ and for all $N \geq N_0$ large enough and $R_N \geq 1$, it follows that
\begin{equation*}
   \pr{}{Y_N \leq (1-\abs{\rho})  (x + \hat{\beta}_N f_N R_N)} \leq \pr{}{\abs{Y_1} \geq (\abs{\rho} - 1) \beta f_N}, \quad N \geq N_0.
\end{equation*}
For large $N$, using the last two inequalities in \eqref{eq:proof_char_fct}, we arrive at
\begin{equation}\label{eq:bound_with_AR_1_and_tail}
p_N(x) \leq (N+1) \pr{}{\abs{Y_1} \geq C_1 f_N} + \pr{}{\abs{Z_{N-1}} \leq 2  \hat{\beta}_N f_N R_N },
\end{equation}
where $C_1 := \min\set{1,(\abs{\rho}-1)\beta}$. Set $\tilde{Z_n} := \rho^{-n}Z_n = \sum_{k=1}^n \rho^{-k} Y_k$. Then
\[
   \pr{}{\abs{Z_{N-1}} \leq 2 \hat{\beta}_N f_N R_N} = \pr{}{\abs{\tilde{Z}_{N-1}} \leq 2 \abs{\rho}^{-(N-1)} \hat{\beta}_N f_N R_N}.
\]
Note that $\tilde{Z}_n$ converges a.s.\ to a random variable $\tilde{Z}_\infty$ by Kolmogorov's Three Series Theorem. Moreover, for $u,v>0$, 
 \begin{align*}
    &\pr{}{\abs{\tilde{Z}_\infty} \leq u+v} \geq \pr{}{\abs{\tilde{Z}_\infty - \tilde{Z}_N} \leq u+v - \abs{\tilde{Z}_N}, \abs{\tilde{Z}_N} \leq u}\\
 & \quad \geq \pr{}{\abs{\tilde{Z}_\infty - \tilde{Z}_N} \leq v, \abs{\tilde{Z}_N} \leq u} = \pr{}{\abs{\tilde{Z}_\infty - \tilde{Z}_N} \leq v}\pr{}{ \abs{\tilde{Z}_N} \leq u}.
 \end{align*}
The last equality follows from the independence of increments of $\tilde{Z}$. Hence,
\[
\pr{}{ \abs{\tilde{Z}_N} \leq u} \le \frac{\pr{}{\abs{\tilde{Z}_\infty} \leq u+v}}{1-\pr{}{\abs{\tilde{Z}_\infty - \tilde{Z}_N} > v}}, \quad u,v>0, N \ge 1.
\]
Using this inequality with $u=v=C_2 \abs{\rho}^{-N} \hat{\beta}_N f_N R_N$, we obtain that
 \begin{align*}
    \pr{}{\abs{\tilde{Z}_{N-1}} \leq 2 \abs{\rho}^{-(N-1)} \hat{\beta}_N f_N R_N} &\leq \frac{\pr{}{\abs{\tilde{Z}_{\infty}} \leq 4 \abs{\rho}^{-(N-1)} \hat{\beta}_N f_N R_N}}{1 - \pr{}{ \abs{ \tilde{Z}_\infty - \tilde{Z}_{N-1} } > 2 \abs{\rho}^{-(N-1)} \hat{\beta}_N f_N R_N }} \\
    &\le  2 \, \pr{}{\abs{\tilde{Z}_{\infty}} \leq 4 \abs{\rho}^{-(N-1)} \hat{\beta}_N f_N R_N}
 \end{align*}
where the last inequality holds for all $N$ sufficiently large in view of the following estimates: Since $R_N \ge 1$, $\hat{\beta}_N \ge \beta > 0$ for large $N$, $\E{}{\abs{Y_1}^a} < \infty$ (w.l.o.g.\ $a \in (0,1)$) and $f_N \to \infty$, we have that
 \begin{align*}
    &\pr{}{ \abs{ \tilde{Z}_\infty - \tilde{Z}_{N-1} } > 2 \abs{\rho}^{-(N-1)} \hat{\beta}_N f_N R_N } = \pr{}{\abs{\sum_{n=N}^\infty \rho^{-n} Y_n}^a > (2 \beta \abs{\rho}^{-(N-1)} f_N)^a } \\
&\le \pr{}{\sum_{n=N}^\infty \abs{\rho}^{-an} \abs{Y_n}^a > (2 \beta \abs{\rho}^{-(N-1)} f_N)^a } \leq \frac{\sum_{n=N}^\infty \abs{\rho}^{-an} \E{}{\abs{Y_1}^a}}{(2 \beta \abs{\rho}^{-(N-1)} f_N)^a} \\
&\leq C_2 \frac{\abs{\rho}^{-aN}}{\abs{\rho}^{-aN} f_N^a} =  C_2 \frac{1}{f_N^a} \to 0.
 \end{align*}
In the first inequality, we have used that $(x+y)^a \le x^a + y^a$ for $x,y \ge 0$ and $a \in (0,1)$.
We have shown that \eqref{eq:bound_with_AR_1_and_tail} implies for all $N$ large enough that
\begin{equation}\label{eq:proof_with_char_fct_end}
p_N(x) \leq (N+1) \pr{}{\abs{Y_1} \geq C_1 f_N} + 2 \, \pr{}{\abs{\tilde{Z}_{\infty}} \leq 2 C_2 \abs{\rho}^{-N} \hat{\beta}_N f_N R_N}.
\end{equation}
If $f_N \to \infty$ is chosen such that $\abs{\rho}^{-N} \hat{\beta}_N f_N R_N \to 0$, we conclude from \eqref{eq:proof_with_char_fct_end} and Proposition~\ref{prop:small-dev-of-series-char-fct} that
\begin{equation}\label{eq:proof_of_exp_decay}
	p_N(x) \leq (N+1)  \pr{}{\abs{Y_1} \geq C_1 f_N} + C_4 \abs{\rho}^{-N} \hat{\beta}_N f_N R_N, \quad N \to \infty.  
\end{equation}
Let us now state the suitable choice for $f_N$. First, recall that by assumption, we have that $\hat{\beta}_N = e^{o(N)}$. \\
Assume first that $\abs{r} \leq 1$. Then $R_N \leq N$. One can set $f_N := \delta^N$ where $1 < \delta < \abs{\rho}$, use Chebychev's inequality (recall that $\E{}{\abs{Y_1}^a} < \infty$) and \eqref{eq:proof_of_exp_decay} to show that
\[
   p_N(x) \precsim N \delta^{-aN} +  \abs{\rho/\delta}^{-N}e^{o(N)} N = e^{o(N)} \, \left( \delta^a \wedge (\abs{\rho}/\delta) \right)^{-N}, \quad N \to \infty.
\]
If $\abs{r} > 1$, $R_N \asymp \abs{r}^N$, take $f_N :=  \delta^N$ where $1 < \delta < \abs{\rho/r}$, and as above, one sees that
\[
   p_N(x) \precsim N \delta^{-aN} +  \abs{\rho/(\delta r)}^{-N}e^{o(N)}  = e^{o(N)} \, \left( \delta^a \wedge (\abs{\rho/(r \delta)} \right)^{-N}, \quad N \to \infty.
\]
If $\E{}{\exp(\abs{Y_1}^\alpha)} < \infty$ for some $\alpha > 0$, it suffices to take $f_N := N^{2/\alpha}$ to obtain 
\[
   p_N(x) \leq (N+1) \E{}{\exp(\abs{Y_1}^\alpha)} \exp(-C_1^\alpha N^2) + C_3 \abs{\rho}^{-N} e^{o(N)} N^{2/\alpha} R_N,
\]
and it is then easy to conclude that $\liminf - N^{-1} p_N(x) \geq - \log(1/\abs{\rho}) = \log(\abs{\rho})$ if $\abs{r} \leq 1$ and $\liminf - N^{-1} p_N(x) \geq \log(\abs{\rho/r})$ if $\abs{r} > 1$. \\
\textbf{Case 2:} Finally, assume that $\beta_n = 0$ for all $n$. Then $X_n = Z_n = \sum_{k=1}^n \rho^{n-k} Y_k$. Let $0 \leq f_N \to \infty$ to be specified later. Clearly, for large $N$,
\begin{align*}
   \pr{}{\sup_{n=1,\dots,N} Z_n \leq x} &\leq \pr{}{Z_{N-1} \leq x, Z_N \leq x} \leq \pr{}{Z_{N-1} \leq f_N, Z_N \leq f_N}\\
&\leq \pr{}{\abs{Z_{N-1}} \leq f_N} + \pr{}{Y_1 \leq (1-\abs{\rho})f_N},
\end{align*}
where we have used \eqref{eq:proof_char_fct_1} in the last inequality. But the last line is just a special case of \eqref{eq:proof_char_fct} with $x=0,\hat{\beta}_N = R_N = 1$, so we can proceed as above.
\end{proof}
We can apply  Theorem~\ref{thm:exp_decay_via_char_fct} to prove that $p_N$ decays exponentially for $(a_1,a_2) \in E_1$, cf.\ Figure~\ref{fig:2}.
\begin{cor}\label{cor:north-west-region}
Let $(a_1,a_2) \in E_1$. Assume that $Y_1$ satisfies the conditions of Theorem~\ref{thm:exp_decay_via_char_fct}. Then there is a constant $C > 0$ such that for every $x \geq 0$, it holds that
\[
  \liminf_{N \to \infty} - N^{-1} \log \pr{}{\sup_{n=1,\dots,N} X_n \leq x} \geq C.
\]
If $\E{}{\exp(\abs{Y_1}^\alpha)} < \infty$ for some $\alpha > 0$, then 
\[
 C \geq 
 \begin{cases}
 \log(\abs{s_2}/s_1), &\quad a_1 + a_2 > 1, \\
 \log \abs{s_2}, &\quad else. 
 \end{cases}
\]
\end{cor}
\begin{proof}
   For $(a_1,a_2) \in E_1$, we have that $s_2 < - 1$ and $\abs{s_2} > s_1 > 0$. Hence, we can apply Theorem~\ref{thm:exp_decay_via_char_fct} with $\rho = s_2$ and $r=s_1$. To get the lower bound on $C$, note that $\abs{r} = s_1 \le 1$ amounts to $a_1 + a_2 \leq 1$. 
\end{proof}
\begin{remark}
   One can show by direct computation that the correlation coefficient $\rho_n$ of $X_{n-1}$ and $X_n$, given by 
\[ 
  \rho_n = \E{}{X_{n-1} X_n}/\sqrt{\E{}{X_{n-1}^2} \E{}{X_n^2}}, 
\] 
satisfies $\rho_n = -1 + O(\abs{s_1/s_2}^n)$. Clearly, $p_N \leq \pr{}{X_{n-1} \leq 0, X_n \leq 0}$, and if $Y_1$ is a centered Gaussian random variable, we get in view of a well-known formula for Gaussian random variables (see e.g.\ Exercise 8.5.1 in \cite{grimmett-stirzaker:2001}) that
\[
   \pr{}{X_{n-1} \leq 0, X_n \leq 0} = \frac{1}{2\pi}\left(\frac{\pi}{2} + \arcsin \rho_n \right).
\]
Since $\pi/2 + \arcsin x \sim \sqrt{2(1+x)}$ as $x \downarrow -1$ (by l'H\^{o}pital's rule), it follows that $p_N \precsim \abs{s_1/s_2}^{N/2}$.
\end{remark}
Note that the previous results do not cover the case $a_1 + 1 = a_2$ if $a_2 \in (0,1)$. Let us now turn to this particular case. One verifies that $c_n = (a_1^{n+1} + (-1)^n)/(a_1 + 1)$, i.e.\ $c_n$ osciallates but does not diverge as in Theorem~\ref{thm:exp_decay_via_char_fct}. We show that $p_N$ still decreases at least exponentially in this case.
\begin{prop}\label{prop:NW-boundary}
   Let $a_1 + 1 = a_2$ and set $Z_n = a_2 Z_{n-1} + Y_n$ for $n \geq 1$. Then, for all $x \geq 0$ and $N \geq 1$,
\[
   \pr{}{\sup_{n=1,\dots,N} X_n \leq x} \leq \pr{}{\sup_{n=1,\dots,N} Z_n \leq 2x}.
\]
In particular, if $a_2 \in (0,1)$, $\E{}{(Y_1^-)^\alpha} < \infty$ and $\pr{}{Y_1 \geq 2x(1- a_2)} > 0$, it holds that $p_N(x) \precsim \exp(-\lambda N)$ for some $\lambda = \lambda(x) > 0$.
\end{prop}
\begin{proof}
   Note that $X_{n+1} + X_n = (a_1 + 1) X_n + a_2 X_{n-1} + Y_{n+1} = a_2 (X_n + X_{n-1}) + Y_{n+1}$. Hence, $(Z_n)_{n \geq 1}$ can be written in the form $Z_n := X_n + X_{n-1}$. In particular, $X_n \leq x$ for $n=1,\dots,N$ implies that $Z_n \leq 2x$ for $n=1,\dots,N$.\\
If $a_2 \in (0,1)$, we deduce from Theorem~\ref{thm:AR_1_nov-kord} that $p_N(x)$ decays exponentially under the conditions stated above.
\end{proof}
In fact, the idea of proof of Proposition~\ref{prop:NW-boundary} can be generalized as follows:
if $X$ is AR($p$), one can try to determine $b_1,b_2 > 0$ such that $(Z_n)_{n \geq 1}$ is AR($p-1$) where $Z_n := b_1 X_n + b_2 X_{n-1}$. Then we always have that $X_n \leq 0$ for $n=1,\dots,N$ implies $Z_n \leq 0$ for $n=1,\dots,N$. We carry this out for $p=2$.
\begin{prop}\label{prop:reduction_to_AR_1}
   Let $a_1^2 + 4 a_2 > 0$. Moreover, assume that either $a_1, a_2 < 0$ or that $a_1 + a_2 < 1$ if $a_2 > 0$. Then $s_2 < 0$, $-a_2/s_2 < 1$ and $Z_n := X_n - s_2 X_{n-1}$ satisfies $Z_n = -a_2/s_2 Z_{n-1} + Y_n$. In particular,
\[
   \pr{}{\sup_{n=1,\dots,N} X_n \leq x} \leq \pr{}{\sup_{n=1,\dots,N} Z_n \leq (1 - s_2) x}, \quad x \geq 0.
\]
\end{prop}
\begin{proof}
Let us determine $b_1,b_2>0$ such that $(Z_n)_{n \geq 1}$ defined by $Z_n := b_1 X_n + b_2 X_{n-1}$ is an AR($1$)-process. We have that
\begin{align*}
   Z_n = (b_1 a_1 + b_2) X_{n-1} + b_1 a_2 X_{n-2} + b_1 Y_n = \frac{b_1 a_1 + b_2}{b_1} b_1 X_{n-1} + \frac{b_1 a_2}{b_2} b_2 X_{n-2} + b_1 Y_n.
\end{align*}
Hence, if $(b_1 a_1 + b_2)/b_1 = b_1 a_2 / b_2$, it follows that
\[
   Z_n = \frac{b_1 a_2}{b_2} Z_{n-1} + b_1 Y_n = \frac{a_2}{\lambda} Z_{n-1} + b_1 Y_n,
\]
where $\lambda := b_2/b_1 > 0$ satisfies $a_1 + \lambda = a_2/\lambda$, i.e.\ $\lambda^2 + a_1 \lambda - a_2 = 0$. The solutions to this equation are $-s_1$ and $-s_2$. Since $a_1^2 + 4a_2 > 0$, we have that $s_2 < s_1$. Hence, we can find $\lambda > 0$ such that $Z$ defines an AR($1$)-process if and if only $s_2<0$, and $\lambda = -s_2$ in that case. Now $s_2 < 0$ amounts to $a_1 \leq 0$ or $a_1,a_2 >0$ since $h > 0$. \\
It follows that 
\[
   \bigcap_{n=1}^N \set{X_n \leq x} \subseteq \bigcap_{n=1}^N \set{Z_n \leq (b_1 + b_2) x} = \bigcap_{n=1}^N \set{Z_n \leq b_1(1 - s_2) x}, \quad x \geq 0.
\]
Finally, $a_2 /\lambda < 1$ if and only if $a_1 + 2a_2 < h$. If $a_1,a_2 >0$, this amounts to $a_1 + a_2 < 1$. \\
In the remaining cases, we necessarily have that $a_1 \leq 0$. If also $a_1 + 2a_2 \leq 0$ (in particular, if $a_1,a_2 \leq 0$), the inequality is obviously satisfied. Finally, if $a_1 + 2a_2 > 0$, $a_1 + 2a_2 < h$ is equivalent to $a_1^2 +4a_1 a_2 + 4 a_2^2 < a_1^2 + 4 a_2$, i.e.\ $a_1 + a_2 < 1$ since $a_2 >0$. The assertion of the proposition follows if we set $b_1 = 1$ and $b_2 = -s_2$.
\end{proof}
The preceding proposition allows us to find exponential upper bounds for the survival probability $p_N$ for a wide class of distributions. Specifically, we obtain exponential upper bounds for certain parameters $a_1$ and $a_2$ and distributions that do not fulfill the requirements of Theorem~\ref{thm:exp_decay_via_char_fct}. Let us record this result as a corollary:
\begin{cor}
   Let $a_1,a_2$ be such that $a_2 > 0$ and $a_1 + a_2 < 1$. Assume that $\E{}{(Y_1^-)^\alpha} < \infty$ for some $\alpha > 0$. Let $x \geq 0$ such that $\pr{}{Y_1 > x (1-s_2)(1-a_2/s_2)} > 0$. Then $p_N(x) \precsim \exp(-\lambda N)$ for some $\lambda = \lambda(x) > 0$.
\end{cor}
\begin{proof}
   Set $\rho := -a_2/s_2$ and let $(Z_n)_{n \geq 1}$ satisfy $Z_n = \rho Z_{n-1} + Y_n$. By Proposition~\ref{prop:reduction_to_AR_1}, we have that $\rho \in (0,1)$ and that $p_N(x) \leq \pr{}{\sup_{n=1,\dots,N} Z_n \leq x(1-s_2)}$. The claim now follows from Theorem~\ref{thm:AR_1_nov-kord}.
\end{proof}

Let us finally turn to the region $a_1 > 0$ and $a_1^2+4a_2 < 0$ ($E_3$ in Figure~\ref{fig:2}) so that the sequence $c_n$ involves expressions with sine and cosine, cf.\  \eqref{eq:sol_diff_eq_sin_cos}. 
\begin{prop}\label{prop:c_k_with_sin_cos}
   Let $(a_1,a_2) \in E_3$. Assume that $\pr{}{Y_1 > 0} > 0$. Then there exists $\lambda > 0$ such that $p_N \precsim \exp(- \lambda N)$ as $N \to \infty$.
\end{prop}
\begin{proof}
   The recursion $X_n = a_1 X_{n-1} + a_2 X_{n-2} +Y_n$ allows us to express $X_n$ as follows ($n \geq k+2$):
\[
   X_n = \alpha_k X_{n-k} + \beta_k X_{n-k-1} + L_k(Y_{n-k+1},\dots,Y_n)
\]
where $L_k(x_1,\dots,x_k)$ is some linear combination of $x_1,\dots,x_k$. Clearly, $\alpha_1 = a_1$, $\beta_1 = a_2$ and $L_1(x_1) = x_1$ and iteratively, we get that $\alpha_{k+1} = a_1 \alpha_k + \beta_k$, $\beta_{k+1} = a_2 \alpha_k$ and $L_{k+1}(x_1,\dots,x_{k+1}) = \alpha_k x_1 + L_k(x_2,\dots,x_{k+1})$ for $k \geq 1.$ In particular, $\alpha_k = a_1 \alpha_{k-1} + a_2 \alpha_{k-2}$ for $k \geq 2$ with $\alpha_0 = 1$ and $\alpha_1 = a_1$, hence,
\[
 \alpha_k = c_k, \quad \beta_k = a_2 c_{k-1}, \quad L_k(x_1,\dots,x_k) = \sum_{j=1}^k c_{k-j} x_j.
\]
Let $q := \inf\set{ k \geq 1 : c_k \leq 0 }$. Assume that $q < \infty$ by \eqref{eq:sol_diff_eq_sin_cos}. Then, if $X_n \leq 0$ for all $n \leq N$, it follows that
\begin{align*}
   0 &\geq X_n = c_q X_{n-q} + a_2 c_{q-1} X_{n-q-1} + L_q(Y_{n-q+1},\dots,Y_n) \\
 &\geq 0 + 0 + L_q(Y_{n-q+1},\dots,Y_n), \quad n = q+2,\dots,N,
\end{align*}
where we have used the fact that $a_2 c_{q-1} < 0$ by the definition of $q$.\\
In particular, we have that
\begin{align*}
   \pr{}{\sup_{n=1,\dots,N} X_n \leq 0} &\leq \pr{}{\sup_{n=q+2,\dots,N} L_q(Y_{n-q+1},\dots,Y_n) \leq 0} \\
					&\leq \pr{}{\sup_{k=1,\dots,\lfloor N/(q+1) \rfloor} L_q(Y_{k(q+1)-q+1},\dots,Y_{k(q+1)}) \leq 0}\\
					&\leq \pr{}{L_q(Y_{2},\dots,Y_{q+1}) \leq 0}^{\lfloor N/(q+1) \rfloor},
\end{align*}
since $(L_q(Y_{kq+1},\dots,Y_{(k+1)q})_{k=0,1,\dots}$ are i.i.d. Next, note that $X_q$ and $L_q(Y_2,\dots,Y_{q+1})$ have the same law. Hence, using that $c_0,\dots,c_{q-1} > 0$ and $\pr{}{Y_1 > 0} > 0$, we have that
\begin{align*}
   \pr{}{X_q > 0} = \pr{}{\sum_{k=1}^q c_{q-k} Y_k > 0} \geq \pr{}{Y_1 > 0}^q > 0.
\end{align*}
It remains to show that $q < \infty$. Let $\varphi \in (0,\pi/2)$ be the angle associated with $(a_1,a_2)$ in \eqref{eq:sol_diff_eq_sin_cos}. Since $a_1 > 0$, it follows from \eqref{eq:sol_diff_eq_sin_cos} that $c_n \leq 0$ for some $n$ if $\sin(n\varphi) \leq 0$ and $\cos(n \varphi) \leq 0$ for some $n$. Take $n = \lceil \pi / \varphi \rceil$. Clearly, $ \pi \leq n \varphi \leq (\pi/\varphi + 1)\varphi \leq 3\pi/2$ since $\varphi \leq \pi/2$. Since $\sin x \leq 0$ and $\cos x\leq 0$ for all $x \in [\pi,3\pi/2]$, we have shown that $q \leq \lceil \pi / \varphi \rceil$.
\end{proof}

We are now ready to give a proof of Theorem~\ref{thm:exp_upper_bound_summary} which is a corollary of the previous results. A look at Figure~\ref{fig:2} will be helpful to distinguish the different cases. 
\begin{proof}(of Theorem~\ref{thm:exp_upper_bound_summary})
On $E_1$, the assertion follows from Corollary~\ref{cor:north-west-region}. On $E_2 = (-\infty,0]^2$, the assertion is trivial. If $(a_1,a_2) \in E_3$, we can apply Proposition~\ref{prop:c_k_with_sin_cos}. The remaining cases covered by Theorem~\ref{thm:summable_seq_1} and Proposition \ref{prop:NW-boundary} (the latter is needed for the strip  $a_2 = 1 + a_1$ with $a_1 \in (-1,0)$ only).
\end{proof}

Note that we have established exponential upper bounds on $p_N$ under various conditions on the distribution of $Y_1$ in the region where $c_n$ goes to $0$ for AR($2$)-processes (cf.\ Remark \ref{rem:c_n_to_zero}) except for the the curve $a_1^2 + 4a_2 = 0$ where $a_1 \in (-2,2)$ and $c_n = (a_1/2)^n (n+1)$. By Theorem~\ref{thm:summable_seq_1}, we know that $p_N \precsim \exp(-\lambda N/\log N)$ in that case if $\E{}{\exp(\abs{Y_1}^\alpha)}$ is finite. If $Y_1$ has a Gaussian law with zero mean, the next proposition establishes an exponential upper bound on $p_N$ in that case. In particular, in combination with the Theorems~\ref{thm:poly_decay_summary}, \ref{thm:exp_upper_bound_summary} and \ref{thm:pos_lim_summary}, we directly obtain Theorem~\ref{thm:gaussian-AR_2-summary}.
\begin{prop}\label{prop:Gaussian_double_root}
   Let $Y_1$ have a Gaussian law. Let $\rho \in (0,1)$ and $(\alpha_n)_{n \geq 0}$ denote a sequence of positive numbers with the following properties
\[
   \alpha_{n+m} \leq C \alpha_n \alpha_m \quad (n,m \geq 0), \qquad \lim_{n \to \infty} e^{-\lambda n} \, \alpha_n = 0 \quad \forall \lambda >0.
\]
Set $X_n := \sum_{k=1}^n \alpha_{n-k} \rho^{n-k} Y_k$. It holds that
\[
   \liminf_{N \to \infty} - N^{-1} \, \log \pr{}{\sup_{n=1,\dots,N} X_n \leq x} > 0, \quad x \in \mathbb{R}.
\]
\end{prop}
\begin{proof}
Clearly, we may suppose that $\E{}{(Y_1 - \E{}{Y_1})^2} = 1$. Moreover, it suffices to consider the case $\E{}{Y_1} = 0$. To see this, set $\sum_{k=1}^n \alpha_{n-k} \rho^{n-k} (Y_k - \mu)$. If $\mu := \E{}{Y_1} < 0$, we have that 
\[
  X_n = \sum_{k=1}^n \alpha_{n-k} \rho^{n-k} (Y_k - \mu) + \mu \sum_{k=0}^{n-1} \alpha_{k} \rho^k \geq \tilde{X}_n + \mu \sum_{k=0}^\infty \alpha_k \rho^k,
\]
where $A := \sum_{k=0}^\infty \alpha_k \rho^k < \infty$ since $\rho < 1$ and $\alpha_n = e^{o(n)}$. Hence,
\[
   \pr{}{\sup_{n=1,\dots,N} X_n \leq x} \leq \pr{}{ \sup_{n=1,\dots,N} \tilde{X}_n \leq x - \mu A}.
\]
Similarly, if $\mu > 0$, $X_n \geq \tilde{X}_n$ for all $n$, and therefore
\[
   \pr{}{\sup_{n=1,\dots,N} X_n \leq x} \leq \pr{}{ \sup_{n=1,\dots,N} \tilde{X}_n \leq x}.
\]
Hence, we can assume from now on that $\E{}{Y_1} = 0$ and $\E{}{Y_1^2} = 1$. Let $\rho < \delta < 1$ and set
\[
   \gamma_n := \sqrt{\frac{\sum_{k=0}^{n-1} \rho^{2k} \alpha_k^2}{\sum_{k=0}^{n-1} \delta^{2k}}}, \qquad Z_n := \gamma_n \sum_{k=1}^n \delta^{n-k} Y_k.
\]
We would like to apply Slepian's inequality (Corollary 3.12 in \cite{ledoux-talagrand:1991}) to compare the probabilities that $X$ and $Z$ stay below $0$ until time $N$. By construction, we have that $\E{}{X_n^2} = \E{}{Z_n^2}$ for all $n \geq 1$. Next, note that $\gamma_n \geq \alpha_0 \sqrt{1-\delta^2}$ for all $n \geq 1$. Hence, if $n>m \geq 1$, we have that
\begin{align*}
  \E{}{Z_n Z_m} = \gamma_n \gamma_m \sum_{k=1}^m \delta^{n-k} \delta^{m-k} \geq \alpha_0^2 (1-\delta^2) \delta^{n-m} \sum_{k=1}^m \delta^{2(m-k)} \geq  C_1 \delta^{n-m},
\end{align*}
where $C_1 := \alpha_0^2 (1-\delta^2)$. Moreover, 
\begin{align*}
   \E{}{X_n X_m} &= \sum_{k=1}^m \alpha_{n-k} \alpha_{m-k} \rho^{m-k} \rho^{n-k} = \rho^{n-m} \sum_{k=1}^m \alpha_{(n-m) + m-k} \alpha_{m-k} \rho^{2(m-k)} \\
&\leq C \rho^{n-m} \alpha_{n-m} \sum_{k=1}^m \alpha_{m-k}^2 \rho^{2(m-k)} \leq C \rho^{n-m} \alpha_{n-m} \sum_{k=0}^\infty \alpha_{k}^2 \rho^{2k} =: C_2 \rho^{n-m} \alpha_{n-m}. 
\end{align*}
In the last equality, we have used that $\sum_{k=0}^\infty \alpha_{k}^2 \rho^{2k}$ converges since $\alpha_n = e^{o(n)}$. Now $C_1 \delta^{n-m} \geq C_2 \alpha_{n-m} \rho^{n-m}$ holds whenever $n-m \geq q$ for some $q \geq 1$ since $\delta > \rho$ and $a_n$ grows slower than any exponential. In particular, $\E{}{X_n X_m} \leq \E{}{Z_n Z_m}$ whenever $\abs{n-m} \geq q$. \\
Hence, using Slepian's inequality, we obtain that
\[
 \pr{}{\sup_{n=1,\dots,N} X_n \leq x} \leq \pr{}{\sup_{n=1,\dots,\lfloor N/q \rfloor} X_{n q} \leq x} \leq \pr{}{\sup_{n=1,\dots,\lfloor N/q \rfloor} Z_{n q} \leq x}.
\]
Let $\tilde{Z}_n := \delta^{-nq} Z_{nq} / \gamma_{nq} = \sum_{k=1}^{nq} \delta^{-k} Y_k$.  One verifies easily that $(\tilde{Z}_n)_{n \geq 1}$ is equal in distribution to $(B(t_n))_{n \geq 1}$ where $(B_t)_{t \geq 0}$ is a one-dimensional Brownian motion and $t_n := \sum_{k=1}^{nq} \delta^{-2k} = C_\delta (\delta^{-2nq} - 1)$, so
\begin{align*}
    p_N &\leq \pr{}{\sup_{n=1,\dots,N} Z_{n q} \leq x} = \pr{}{\bigcap_{n=1}^N \set{ \tilde{Z}_n \leq x \delta^{-nq}/\gamma_{nq} } } \\
&= \pr{}{ \bigcap_{n=1}^N \set{ B( C_\delta(\delta^{-2qn} - 1) ) \leq x \delta^{-nq}/\gamma_{nq} } } \leq \pr{}{\sup_{n=1,\dots,N} B( \delta^{-2qn} - 1 ) \leq \tilde{x} },
\end{align*}
where we have used the scaling property of Brownian motion and the fact that $\gamma_n \geq C_1 \alpha_0/\delta^n$ for all $n$ (i.e.\ $\tilde{x} := x / (C_1 \alpha_0 C_\delta^{1/2})$). Next, note that 
\begin{align*}
   \pr{}{\sup_{n=1,\dots,N} B(\delta^{-2qn}) \leq 0} &\geq \pr{}{B_1 \leq -\tilde{x}, \sup_{n=1,\dots,N} B(\delta^{-2qn}) - B_1 \leq \tilde{x}} \\
&= \pr{}{B_1 \leq - \tilde{x}}  \, \pr{}{\sup_{n=1,\dots,N} B(\delta^{-2qn} - 1) \leq \tilde{x}}.
\end{align*}
An application of Slepian's inequality together with a subadditivity argument (see e.g.\ Eq.\ 2.6 of \cite{aurzada-baumgarten:2011}) yields that
\[
   \liminf_{N \to \infty} N^{-1} \log \pr{}{\sup_{n=1,\dots,N} B(a^n) \leq 0} > 0, \quad a > 1.
\]
\end{proof}
\subsection{Exponential lower bounds}
Let us now comment on exponential lower bounds for AR-processes. In general, we cannot expect to find exponential lower bounds in the whole region where we have established exponential upper bounds. The following example illustrates this point for AR($2$)-processes.
\begin{example}
   If $X$ is AR($p$) and the innovation $Y_1$ takes only the values $\pm y$ for some $y>0$ and $a_1 < - 1$, then $p_2 = \pr{}{X_1 \leq 0,X_2 \leq 0} = 0$. Indeed, on $\set{X_1 \leq 0} = \set{Y_1 = -y}$, we have that $X_2 = a_1 Y_1 + Y_2 \geq -ya_1 - y = -y(a_1 + 1) > 0$. \\
Similarly, if $a_1 \in [-1,0]$ and $a_1(a_1 + 1) + a_2 < -1$, one has that $p_3 = 0$.
\end{example}
Let us also remark that if $X$ is AR($p$) with $a_1 \geq 0, \dots, a_p \geq 0$, it is trivial to obtain the exponential lower bound $p_N(x) \ge p_N \ge \pr{}{Y_1 \leq 0}^N$.\\
The following theorem states a simple condition on the coefficients $a_1,\dots,a_p$ such that the survival probability cannot decay faster than exponentially.
\begin{thm}\label{thm:AR_p_exp_lower_bound}
   If $X$ is AR($p$) with $\sum_{k=1}^p \abs{a_k} < 1$, it holds that $p_N \succsim c^N$ for all $N$ where $c \in (0,1)$. Moreover, if $a_k > 0$ for some $k \in \set{1,\dots,p}$, one may take 
\[
   c := \sup \set{ \pr{}{Y_1 \in [\alpha(1 - a_+),\alpha \abs{a_-}]} : \alpha < 0}
\]
where (with the convention that $\sum_{\emptyset} = 0$)
\[
   a_+ := \sum_{k \in I_+} a_k, \quad a_- :=  \sum_{k \in I_-} a_k, \quad I_+ = \set{k : a_k > 0}, \quad I_- = \set{k : a_k < 0}. 
\]
\end{thm}
\begin{proof}
The goal is to find intervals $([\alpha_n,\beta_n])_{n \geq 1}$ such that 
\begin{equation}\label{eq:proof_inclusion_exp_lower_bound}
   \bigcap_{k=1}^n \set{Y_k \in [\alpha_k,\beta_k]} \subseteq  \bigcap_{k=1}^n \set{X_k \in [\gamma_k,0]}, \quad n \geq 1.
\end{equation}
If \eqref{eq:proof_inclusion_exp_lower_bound} holds and $\pr{}{Y_n \in [\alpha_n,\beta_n]} \geq c > 0$ for all $n \geq N_0$, we immediately obtain that $p_N \succsim c^N$.\\
Using the recursive definition of $X$, we can iteratively define the sequences $(\alpha_n)_{n \geq 1}$, $(\beta_n)_{n \geq 1}$,$ (\gamma_n)_{n \geq 1}$ as follows: Start with $\gamma_1 = \alpha_1 < \beta_1 \leq 0$. Define successively (with the convention $\gamma_n = 0$ for $n \leq 0$)
\[
   \beta_k := - \sum_{j \in I_-} a_j \gamma_{k-j}, \quad \alpha_k < \beta_k, \quad \gamma_k := \sum_{j \in I_+} a_j \gamma_{k-j} + \alpha_k.
\]
It is clear that $\gamma_k \leq 0$ and $\beta_k \leq 0$ for all $k$. We claim that \eqref{eq:proof_inclusion_exp_lower_bound} holds for this choice of $(\alpha_n),(\beta_n)$ and $(\gamma_n)$. For $n=1$, this is obvious, and inductively, if the statement holds for some $n-1 \geq 1$, we have that
\[
   X_n = \sum_{j=1}^p a_k X_{n-j} + Y_n \leq \sum_{j \in I_-}^p a_j X_{n-j} + \beta_n \leq \sum_{j \in I_-} a_j \gamma_{n-j} + \beta_n = 0,
\]
and 
\[
   X_n = \sum_{j=1}^p a_k X_{n-j} + Y_n \geq \sum_{j \in I_+}^p a_j X_{n-j} + \alpha_n \geq \sum_{j \in I_+} a_j \gamma_{n-j} + \alpha_n = \gamma_n. 
\]
Note that the above inequalities hold even if $I_+ = \emptyset$ or if $I_- = \emptyset$. Fix $\alpha_1=\gamma_1 < \beta_1 = 0$ and let $\alpha_k = -\alpha_1 (a_+ - 1)$ for all $k \geq 2$. We claim that $\gamma_k \geq \alpha_1$. Inductively, if the claim holds for all $k \leq n-1$, we have that
\[
   \gamma_n = \sum_{j \in I_+} a_j \gamma_{n-j} -\alpha_1 (a_+ - 1) \geq \alpha_1 a_+ - \alpha_1 (a_+ - 1) = \alpha_1.
\]
It follows that $\beta_n \geq - \alpha_1 a_-$ and in particular, $\alpha_k < \beta_k$ since 
\begin{align*}
   \alpha_k  - \beta_ k \leq -\alpha_1 (a_+ - 1) + \alpha_1 a_- = -\alpha_1 \left( \sum_{k=1}^p \abs{a_k} - 1 \right) < 0.
\end{align*}
In view of \eqref{eq:proof_inclusion_exp_lower_bound}, we obtain that
\begin{align*}
    \pr{}{\sup_{k=1,\dots,n} X_k \leq 0} &\geq \prod_{k=1}^n \pr{}{Y_k \in [\alpha_k,\beta_k]} \\ 
&\geq \pr{}{Y_1 \in [\alpha_1,0]} \pr{}{Y_1 \in [-\alpha_1(a_+ - 1),-\alpha_1 a_-]}^{n-1} 
\end{align*}
\end{proof}
\begin{remark}
   In general, there is no reason to believe that the lower bound of Theorem~\ref{thm:AR_p_exp_lower_bound} is sharp.
\end{remark}
\begin{cor}
Let $(Y_n)_{n \geq 0}$ be a sequence of i.i.d.\ standard Gaussian random variables. Using the notation of Theorem~\ref{thm:AR_p_exp_lower_bound}, if $I_-$ and $I_+$ are nonempty, we have that $p_N \geq c^N$ where
\[
	c = \pr{}{ \alpha^* (1-a_+) \leq Y_1 \leq \alpha^* \abs{a_-} } = \pr{}{ -\sqrt{\frac{- \log A^2}{1-A^2}} \leq Y_1 \leq  - A \sqrt{\frac{- \log A^2}{1-A^2}} } 
\]
and
\[
	\alpha^* := -\sqrt{\frac{\log(1-a_+)^2 - \log \abs{a_-}^2}{(1-a_+)^2- \abs{a_-}^2}} < 0, \quad A := \frac{\abs{a_-}}{1- a_+} \in (0,1).
\]
\end{cor}
\begin{proof}
By Theorem \ref{thm:AR_p_exp_lower_bound}, we have to determine 
\[ 
	\sup_{ \alpha \leq 0} \pr{}{ \alpha(1-a_+) \leq Y_1 \leq \alpha\abs{a_-} } = \sup_{ \alpha \leq 0} \set{ \Phi(\alpha\abs{a_-}) - \Phi(\alpha (1-a_+))},
\]
where $\Phi$ is the cdf of a standard normal random variable. It is not hard to verify that the unique maximum is attained at 
\[
 \alpha^* := - \sqrt{\frac{\log(1-a_+)^2 - \log \abs{a_-}^2}{(1-a_+)^2- \abs{a_-}^2}} < 0.
\]
\end{proof}
\section{Polynomial order}\label{sec:poly}
If $X$ is an AR($2$)-process and $\E{}{Y_1} = 0$, it is known that $p_N$ decays polynomially if $X$ is a centered random walk ($a_1 = 1, a_2 = 0$) or an integrated random walk ($a_1 = 2,a_2 = -1$) under suitable moment conditions. To be more precise, if $S_n = \sum_{k=1}^n Y_k$ is a random walk and $\E{}{Y_1} = 0$, it holds that
\[
   \pr{}{\sup_{n=1,\dots,N} S_k \leq 0} = N^{-(1 -\theta) + o(1)} \quad \text{ for some } \theta \in (0,1) \Longleftrightarrow \pr{}{S_N \leq 0} \to \theta \in (0,1),
\]
see e.g.\ \cite{aurzada-simon:2012}. Moreover, the process $X_n = 2 X_{n-1} - X_{n-1} + Y_n$ is given by $X_n = \sum_{k=1}^n (n-k+1) Y_k = \sum_{k=1}^n S_k$ where $(S_n)_{n \geq 1}$ is the usual random walk. $X$ is called integrated random walk (IRW). Several authors have studied the asymptotic behaviour of $p_N$ in that case if $\E{}{Y_1} = 0$. We refer to the recent article of \cite{dembo-ding-gao:2012} and the references therein. In particlar, it is shown in \cite{dembo-ding-gao:2012} that $p_N \asymp N^{-1/4}$ if $\E{}{Y_1} = 0$ and $\E{}{Y_1^2} \in (0,\infty)$.
\subsection{Integrated processes}
In this subsection, we will prove that $p_N = N^{-1/2 + o(1)}$ under suitable moment conditions if $a_1 + a_2 = 1$ and $\abs{a_2} < 1$. As we will see shortly, these AR(2)-processes can be written as integrated AR(1)-processes. \\ 

Let us begin by characterizing the behaviour of the sequence $(c_n)_{n \geq 0}$ for such $a_1,a_2$. Instead of manipulating the explicit expression for $c_n$ to determine these values of $a_1$ and $a_2$, we give a short proof of the following lemma.
\begin{lemma}\label{lem:diff_eq_sum_eq_1}
    The sequence $(c_n)$ converges to a constant $c \neq 0$ if and only if $a_1 + a_2 = 1$ and $\abs{a_2} < 1$. In that case, $\lim_{n \to \infty} c_n = 1/(1+a_2)$. Moreover, if $a_1 + a_2 = 1$,  $c_n = (1 - (-a_2)^{n+1})/(1 + a_2)$ if $a_2 \neq -1$ and $c_n = n+1$ if $a_2=-1$ for $n \geq 0$.
 \end{lemma}
\begin{proof}
Assume that $a_1 + a_2 = 1$. Then $c_{n+1} = (a_1 + a_2 - a_2) c_n + a_2 c_{n-1} = c_n - a_2(c_n - c_{n-1})$, i.e.\ $c_{n+1} - c_n = -a_2 (c_n - c_{n-1})$. Iteration yields $c_{n+1} - c_n = (-a_2)^n (c_1 - c_0)= (-a_2)^{n+1}$. Hence,
 \[
    c_n = 1 + \sum_{k=1}^n (c_k - c_{k-1}) = 
\begin{cases}
   1 + \sum_{k=1}^n (-a_2)^k = \frac{1 - (-a_2)^{n+1} }{1-(-a_2)}, &\quad a_2 \neq -1, \\
  n+1, &\quad a_2 = -1,
\end{cases}
 \]
and therefore, $c_n \to c = 1/(1+a_2) \neq 0$ if and only if $\abs{a_2} < 1$. On the other hand, if $\lim c_n = c \neq 0$, then the recursion equation implies that $c = a_1c + a_2 c$, i.e.\ $a_1 + a_2 = 1$. By the preceding lines, convergence implies that $\abs{a_2} < 1$.
\end{proof} 
In particular, the preceding lemma shows that
\[
   X_n = \frac{1}{1 + a_2} \left( \sum_{k=1}^n Y_k - \sum_{k=1}^n (-a_2)^{n-k+1} Y_k \right), \quad n \geq 1,
\]
and since $\abs{a_2} < 1$, one expects that the behaviour of $X$ is similar to that of a random walk.\\
Moreover, AR($2$)-processes with $a_1 + a_2 = 1$ and $\abs{a_2} < 1$ can also be regarded as integrated AR($1$)-processes. Let us explain this in more detail.\\
If $\tilde{X}$ is AR($p$) with coefficients $a_1,\dots,a_p$, set $X_n := \sum_{k=1}^n \tilde{X}_k$. 
\begin{align*}
   X_n &= X_{n-1} + \sum_{k=1}^p a_k \tilde{X}_{n-k} + Y_n = X_{n-1} + \sum_{k=1}^p a_k (X_{n-k} - X_{n-k-1}) + Y_n \\
&= (1 + a_1) X_{n-1} + \sum_{k=2}^{p} (a_k - a_{k-1}) X_{n-k} -a_p X_{n-p-1} + Y_n, 
\end{align*}
i.e.\ $X$ is AR($p+1$) and the transfomation of the coefficients  $T_p \colon \mathbb{R}^p \to \mathbb{R}^{p+1}$ is given by
\begin{equation}\label{eq:coeff_of_int}
   T_p(a_1,\dots,a_p) = (a_1 + 1,a_2 - a_1,\dots,a_p - a_{p-1}, - a_p).
\end{equation}
Note that $T_p$ is one-to-one and that $T_p(\mathbb{R}^p)$ is an affine subspace of $\mathbb{R}^{p+1}$. \\
Now, if $\tilde{X}$ is AR(1) with $\tilde{X}_n = \rho \tilde{X}_{n-1} + Y_n$, we have that $X$ with $X_n = \sum_{k=1}^n \tilde{X}_k$ is AR(2) with coefficients $T_1(\rho) = (\rho - 1,-\rho) =: (a_1,a_2)$. In other words, AR(2)-processes with $a_1 + a_2 = 1$ and $\abs{a_2} < 1$ are integrated AR(1)-processes with $\abs{\rho} < 1$. \\
The next theorem states conditions under which the survival probability of an integrated process behaves like $N^{-1/2 + o(1)}$. 
\begin{thm}\label{thm:int_of_summable_seq}
Assume that $\E{}{Y_1} = 0$. Let $\tilde{X}_n = \sum_{k=1}^n \tilde{c}_{n-k} Y_k$ where $\sum_{k=1}^\infty k \abs{\tilde{c}_k} < \infty$ and $\sum_{k=0}^\infty \tilde{c}_k \neq 0$. Set $X_n := \sum_{k=1}^n \tilde{X}_k$. 
\begin{enumerate}
   \item If $\abs{Y_1} \leq M < \infty$ a.s., there is $x_0 \geq 0$ such that for all $x \geq x_0$, it holds that
\[
   \pr{}{\sup_{n=1,\dots,N} X_n \leq x} \asymp N^{-1/2}, \quad N \to \infty.
\] 
\item If $\E{}{\exp(\abs{Y_1}^\alpha)} < \infty$, it holds for all $x \geq 0$ that
\[
   \pr{}{\sup_{n=1,\dots,N} X_n \leq x} \precsim N^{-1/2} (\log N)^{1/\alpha}, \quad N \to \infty.
\]
\item If $\E{}{\exp(\abs{Y_1}^\alpha)} < \infty$ and $\sum_{k=0}^n \tilde{c}_k \geq 0$ for all $n \geq 0$, it holds for all $x \geq 0$ that
\[
    \pr{}{\sup_{n=1,\dots,N} X_n \leq x} \succsim N^{-1/2} (\log N)^{-1/\alpha + o(1)}, \quad N \to \infty.
\]
\end{enumerate}
\end{thm}
\begin{proof}
  First, note that 
\[
   X_n = \sum_{k=1}^n \sum_{j=1}^k \tilde{c}_{k-j}Y_j = \sum_{j=1}^n Y_j \sum_{k=j}^n \tilde{c}_{k-j} = \sum_{j=1}^n Y_j \sum_{k=0}^{n-j} \tilde{c}_k = \sum_{k=1}^n c_{n-k} Y_k
\]
where $c_n := \sum_{k=0}^n \tilde{c}_k \to c = \sum_{k=0}^\infty \tilde{c}_k \neq 0$. Set $S_n := \sum_{k=1}^n c Y_k$, so that for all $n \geq 1$,
\begin{equation*}
   \abs{S_n - X_n} = \abs{\sum_{k=1}^{n} (c - c_{n-k}) Y_k },
\end{equation*}
In particular, if $\abs{Y_1} \leq M < \infty$ a.s., it follows that 
\[
   \abs{S_n - X_n} \leq M \sum_{k=0}^{n-1} \abs{c - c_k } \leq M \sum_{k=0}^{n-1} \sum_{j=k+1}^\infty \abs{\tilde{c}_j} = M \sum_{j=1}^\infty j \abs{\tilde{c}_j} =: \tilde{M} < \infty.
\]
Hence, we get for $x \geq \tilde{M}$ that
\[
   \pr{}{\sup_{n=1,\dots,N} S_n \leq 0} \leq \pr{}{\sup_{n=1,\dots,N} X_n \leq x}  \leq \pr{}{\sup_{n=1,\dots,N} S_n \leq x + \tilde{M}},
\]
and the proof of part 1.\ is complete since $S$ is a centered random walk with finite variance.\\
The proof of part 2.\ is similar. Let $E_N := \set{\abs{Y_k} \leq (2 \log N)^{1/\alpha}, k=1,\dots,N}$. On $E_N$, we get as above that
\begin{equation}\label{eq:diff_X_int_X}
    \abs{S_n - X_n} \leq (2 \log N)^{1/\alpha} \sum_{k=0}^{n-1} \abs{c - c_k} \leq (2 \log N)^{1/\alpha} \sum_{j=1}^\infty j \abs{\tilde{c}_j} =: C (\log N)^{1/\alpha}.
\end{equation}
Hence,
\begin{align*}
   \pr{}{\sup_{n=1,\dots,N} X_n \leq x} &\leq \pr{}{E_N^c} + \pr{}{\sup_{n=1,\dots,N} S_n \leq x + C ( \log N)^{1/\alpha}}.
\end{align*}
By Chebyshev's inequality, 
\[
   \pr{}{E_N^c} \leq N \pr{}{\abs{Y_1} \geq (2 \log N)^{1/\alpha}} \leq N \E{}{\exp(\abs{Y_1}^\alpha)} N^{-2} \asymp N^{-1}.
\]
Finally, by Lemma~\ref{lem:RW_below_f_N} below, it holds that 
\[
   \pr{}{\sup_{n=1,\dots,N} S_n \leq x + C \, (\log N)^{1/\alpha}} \precsim (\log N)^{1/\alpha} \, N^{-1/2},
\]
which proves part 2.\\
It suffices to prove the lower bound of part 3 for $x=0$. Moreover, we use that independent random variables $Y_1,\dots,Y_N$ are associated for every $N$, cf.\ \cite{esary-proschan-walkup:1967}. Since $c_n = \sum_{k=0}^n \tilde{c}_k \geq 0$ for every $n$ by assumption, the function
\begin{align*}
   f_{K,L}(x_1,\dots,x_N) \mapsto 
\begin{cases}
   -1, &\quad \sum_{k=1}^n c_{n-k} x_k \leq 0 \quad \text{for all } n=K,\dots,L\\
  0, &\quad \text{else},
\end{cases}
\end{align*}
is nondecreasing in every component. Hence, the very definition of associated random variables implies for $1 \leq N_0 < N$ that 
\[
   \mathrm{cov}\left(f_{1,N_0}(Y_1,\dots,Y_N), f_{N_0+1,N}(Y_1,\dots,Y_N)\right) \geq 0,
\]
or equivalently,
\[
   \pr{}{\sup_{n=1,\dots,N} X_n \leq 0} \geq  \pr{}{\sup_{n=1,\dots,N_0} X_n \leq 0}  \pr{}{\sup_{n=N_0+1,\dots,N} X_n \leq 0}.
\]
Hence, we can bound the survival probability $p_N$ of $X$ from below as follows:
\begin{align}
   p_N &\geq p_{N_0} \cdot \pr{}{\sup_{n=N_0+1,\dots,N} X_n \leq 0, E_N} \notag \\
&\geq p_{N_0} \cdot \pr{}{\sup_{n=N_0+1,\dots,N} S_n \leq -C \, (\log N)^{1/\alpha}, E_N}. \label{eq:prelim_bound_0}
\end{align}
Note that we have used \eqref{eq:diff_X_int_X} in the second inequality. Next,
\begin{align*}
&\pr{}{\sup_{n=N_0+1,\dots,N} S_n \leq -C \, (\log N)^{1/\alpha}, E_N} \\
&\quad \geq \pr{}{\sup_{n=N_0+1,\dots,N} S_n \leq - C \, (\log N)^{1/\alpha}} - \pr{}{E_N^c}  \\
&\quad \geq \pr{}{\sup_{n=N_0+1,\dots,N} S_n - S_{N_0} \leq 0, S_{N_0} \leq - C \, (\log N)^{1/\alpha}} - \pr{}{E_N^c} \\
&\quad \geq \pr{}{\sup_{n=1,\dots,N} S_n \leq 0} \pr{}{ S_{N_0} \leq - C \, (\log N)^{1/\alpha}} - \pr{}{E_N^c}.
\end{align*}
Let $N_0 := \lfloor \log N \rfloor^{2/\alpha}$. Then $\pr{}{ S_{N_0} \leq - C \, (\log N)^{1/\alpha}} \geq \pr{}{S_{N_0}/\sqrt{N_0} \leq -C}$ and the r.h.s.\ converges to a constant by the CLT. Using the estimate on $\pr{}{E_N^c}$ from above and \eqref{eq:prelim_bound_0}, we have for $N$ large enough that
\begin{equation}\label{eq:prelim_bound_1}
   p_N \geq C_1 \, p_{N_0} \cdot N^{-1/2} = C_1 \, \pr{}{\sup_{n=1,\dots,\lfloor \log N \rfloor^{2/\alpha}} X_n \leq 0} N^{-1/2}.
\end{equation}
Since $c_n \geq 0$ for all $n$, we can now use the trivial estimate $p_{N_0} \geq \pr{}{Y_1 \leq 0}^{N_0} = e^{-\kappa N_0}$ implying for $N$ large enough that
\begin{equation*}
   p_N \geq C_1 \, \exp( - \kappa \lfloor \log N \rfloor^{2/\alpha} ) \, N^{-1/2}.
\end{equation*}
Using this as an a priori estimate for $p_{N_0}$, we get for large $N$ in view of \eqref{eq:prelim_bound_1} that 
\begin{align*}
   p_N &\geq C_1^2 \, \exp( - \kappa \lfloor \log N_0 \rfloor^{2/\alpha} ) \, N_0^{-1/2} \, N^{-1/2} \\
&= C_1^2 \exp\left( - \kappa \lfloor \log \left( \lfloor \log N \rfloor^{1/\alpha} \right) \rfloor^{2/\alpha} \right) \, \lfloor \log N \rfloor^{-1/\alpha} \, N^{-1/2} \\
& \geq C_2 \exp( - C_3 (\log \log N)^{2/\alpha} ) \, (\log N)^{-1/\alpha} \, N^{-1/2}.
\end{align*}
Using this improved estimate again to obtain a lower bound on $p_{N_0}$, we deduce from \eqref{eq:prelim_bound_1} that $p_N \succsim (\log N)^{-1/\alpha + o(1)} N^{-1/2}$.
\end{proof}
\begin{remark}
   One cannot expect to get a useful lower bound without any restriction on the weights $c_n$. For instance, if $Y_1$ takes only values $\pm 1$ and $X_n = \sum_{k=1}^n c_{n-k} Y_k$ with $c_0 = 1, c_1 = -3$, then $\pr{}{X_1 \leq 0, X_2 \leq 0} = \pr{}{X_1 \leq 0, X_1 + X_2 \leq 0} = 0$ . 
\end{remark}
In order to complete the proof of Theorem~\ref{thm:int_of_summable_seq}, let us prove the following lemma.
\begin{lemma}\label{lem:RW_below_f_N}
Let  $(f_n)_{n \ge 1}$ denote a sequence of positive numbers with $f_N \to \infty$ and $f_N/\sqrt{N} \to 0$ as $N \to \infty$. Let $(S_n)_{ n \geq 1}$ denote a centered random walk with $\E{}{S_1^2} \in (0,\infty)$ and let $M_n := \max \set{S_1,\dots,S_n}$. There are a constants $C, N_0$ independent of the sequence $(f_n)$ such that
\[
 \pr{}{ M_N \leq f_N } \le C f_N \, N^{-1/2} , \qquad f_N,N \ge N_0.
\]
\end{lemma}
\begin{proof}
   Since independent random variables are associated (\cite{esary-proschan-walkup:1967}), we have for $1 \le N_0 < N$ that
\begin{align*}
   \pr{}{S_n \le 0, \forall n=1,\dots,N} \ge \pr{}{S_n \le 0, \forall n=1,\dots,N_0} \pr{}{ S_n \le 0, \forall n=N_0 + 1,\dots,N }.
\end{align*}
Now 
\begin{align*}
   &\pr{}{ \sup_{n=N_0 + 1,\dots,N} S_n \le 0 } \ge \pr{}{S_{N_0} \le -f_N, \sup_{n=N_0 + 1,\dots,N} S_n - S_{N_0} \le f_N} \\
&\quad = \pr{}{S_{N_0} \le -f_N} \pr{}{\sup_{n=1,\dots,N-N_0} S_n \le f_N} \ge  \pr{}{S_{N_0} \le -f_N} \pr{}{M_N \le f_N}.
\end{align*}
Hence, we get that
\[
   \pr{}{M_N \le f_N} \le \frac{\pr{}{M_N \le 0}}{\pr{}{M_{N_0} \le 0} \pr{}{S_{N_0} \le -f_N} }.
\]
With $N_0 = \lfloor f(N) \rfloor^2$, it follows from the CLT that $\pr{}{S_{N_0} \le -f_N} \to \pr{}{Z \le -1}$ where $Z$ is a centered Gaussian with variance $\E{}{Y_1^2}$. Moreover, since $\pr{}{M_N \le 0} \sim c N^{-1/2}$, we conclude that
\[
   \frac{\pr{}{M_N \le 0}}{\pr{}{M_{N_0} \le 0} \pr{}{S_{N_0} \le -f_N} } \sim \frac{N^{-1/2}}{N_0^{-1/2} \pr{}{Z \le -1}} \sim f_N N^{-1/2} / \pr{}{Z \le -1}.
\]
\end{proof}
\begin{cor}\label{cor:decay_like_RW}
   Assume that $\E{}{Y_1} = 0$. Let $a_1+a_2=1$ with $\abs{a_2} < 1$ and $x \geq 0$.
\begin{enumerate}
   \item If $\abs{Y_1} \leq M$ a.s., it holds that $p_N(x) \asymp N^{-1/2}$ as $N \to \infty$.
  \item If $\E{}{\exp(\abs{Y_1}^\alpha)} < \infty$ for some $\alpha > 0$, it holds that $p_N(x) = N^{-1/2+o(1)}$ as $N \to \infty$.
\end{enumerate}
\end{cor}
\begin{proof}
   If $X$ is AR(2) with coefficients $a_1,a_2$ as in the statement of the corollary, we have seen that $X_n = \sum_{k=1}^n Z_k$ where
$Z$ is AR(1) with $Z_n = - a_2 Z_{n-1} + Y_n$, i.e.\ $Z_n = \sum_{k=1}^n (-a_2)^{n-k} Y_k$. Since $\sum_{k=0}^n (-a_2)^k > 0$ for all $n$, it is not hard to see that part 2 and part 3 of Theorem~\ref{thm:int_of_summable_seq} imply part 2 of the corollary. Similarly, by part 1 of Theorem~\ref{thm:int_of_summable_seq} and the fact that $p_N(x) \asymp p_N$ (see the comment at the end of Section~\ref{sec:AR-2}), we obtain part 1 of the corollary.  
\end{proof}
In analogy to the results for random walks, it is very likely that the assertion of Corollary~\ref{cor:decay_like_RW} remains true under the much weaker integrability assumption $\E{}{Y_1^2} \in (0,\infty)$. Depending on the sign of $a_1$, we can improve the preceding corollary by proving an upper or lower bound of order $N^{-1/2}$:
\begin{prop}\label{prop:reduction_to_RW}
   Let $a_1 + a_2 = 1$ with $\abs{a_2} < 1$. Assume that $\E{}{Y_1} = 0$, $\E{}{Y_1^2} \in (0,\infty)$. 
\begin{enumerate}
   \item If $a_2 > 0$, we have that $p_N(x) \precsim N^{-1/2}$ for all $x \geq 0$.
  \item If $a_2 < 0$, we have that $p_N(x) \succsim N^{-1/2}$ for all $x \geq 0$.
\end{enumerate}
\end{prop}
\begin{proof}
   For $n \geq 1$, set $S_n := X_n + a_2 X_{n-1}$ and note that
\[
   S_n = a_1 X_{n-1} + a_2 X_{n-2} + Y_n +  a_2 X_{n-1} = X_{n-1} + a_2 X_{n-2} + Y_n = S_{n-1} + Y_n,
\]
i.e.\ $(S_n)_{n \geq 1}$ defines a centered random walk. Moreover, since $a_1 + a_2 = 1$, we have, for $n \geq 1$, that
\[
   X_n = (a_1 - 1) X_{n-1} + X_{n-1} + a_2 X_{n-2} + Y_n = (a_1 - 1) X_{n-1} + S_{n-1} + Y_n = - a_2 X_{n-1} + S_n.
\]
In particular, if $a_2 > 0$, it holds that $X_n \leq x$ for $n=1,\dots,N$ implies that $S_n \leq a_2 x$ for $n=1,\dots,N$ and therefore,
\[
   p_N(x) \leq \pr{}{\sup_{n=1,\dots,N} S_n \leq a_2 x} \precsim N^{-1/2}.
\]
Similarly, if $a_2 < 0$, $S_n \leq 0$ for $n=1,\dots,N$ implies that $X_n \leq 0$ for $n=1,\dots,N$, which yields the lower bound.
\end{proof}
Let us finally remark that Theorem~\ref{thm:int_of_summable_seq} is also applicable to integrated AR($p$)-processes such that the roots $s_1,\dots,s_p$ of the corresponding characteristic polynomial lie inside the unit disc. Let us just state the simplest case of bounded innovations $Y_n$. Set 
\[
   \Delta_p := \set{ (a_1, \dots, a_p) : \max_{k=1,\dots,p} \abs{s_k} < 1 }
\]
where $s_1,\dots,s_p$ are the roots of the characteristic polynomial, see p.\ \pageref{def:char_pol}.
\begin{cor}\label{cor:AR_p_int_summable}
   Let $X$ be the AR($p$)-process corresponding to $(a_1,\dots,a_p) \in \Delta_p$. Assume that $\abs{Y_1} \leq M < \infty$ a.s. Then there is $x_0 \geq 0$ such that for all $x \geq x_0$, we have that
\[
   \pr{}{\sup_{n=1,\dots,N} \sum_{k=1}^n X_k \leq x} \asymp N^{-1/2}.
\]
\end{cor}
Since we know the region $\Delta_2$ explicitly (cf.\ Figure~\ref{fig:c_n_to_zero}), we obtain the following result for AR(3)-processes:
\begin{cor}\label{cor:AR_3_RW}
Let $X$ be AR(3) with $a_1,a_2,a_3$ satisfying
\[
   a_1 + a_2 + a_3 = 1, \quad a_2 < \min \set{1, 3-2a_1}, \quad a_2 > -a_1.
\]
Assume that $\abs{Y_1} \leq M$ a.s.\ for some $M < \infty$. Then there is $x_0 \geq 0$ such that $p_N(x) \asymp N^{-1/2}$ for all $x \geq x_0$.
\end{cor}
\begin{proof}
Let us show that $X$ is an integrated AR(2)-process $\tilde{X}$ with parameters in $\Delta_2$. Since $a_1 + a_2 + a_3 = 1$, we have that $T_2(a_1 - 1, a_1 + a_2 - 1) = (a_1,a_2,a_3)$ where $T_2$ was defined in \eqref{eq:coeff_of_int}. Hence, by Corollary~\ref{cor:AR_p_int_summable}, we only need to show that 
\[
   (a_1 - 1, a_1 + a_2 - 1) \in \Delta_2 = \set{(\tilde{a}_1,\tilde{a}_2) : \tilde{a}_1 + \tilde{a}_2 < 1, \tilde{a}_2 < 1 + \tilde{a}_2, \tilde{a}_2 > -1 },
\]
(see Remark~\ref{rem:c_n_to_zero}) whenever $(a_1,a_2,a_3)$ satisfy the constraints stated in the corollary. Let $\tilde{a}_1 = a_1 - 1$ and $\tilde{a}_2 = a_1 + a_2 - 1$. Now $a_2 < 3 - 2a_1$ amounts to $\tilde{a}_1 + \tilde{a}_2 = 2a_1+a_2 - 2 < 1$. Next, $\tilde{a}_2 < 1 + \tilde{a}_1$ is equivalent to $a_2 < 1$, whereas $\tilde{a}_2 > -1$ translates into $a_1 > - a_2$.
\end{proof}
\subsection{The case $a_1 = 0$}
We still have to consider the case $X_n = X_{n-2} + Y_n$ which is a special case of the equation $X_n = \rho X_{n-2} + Y_n$. The solution of the latter equation is given by 
\[
X_n = 
	\begin{cases}
		\sum_{j=1}^k \rho^{k-j} Y_{2j - 1}, &\quad n=2k-1, k \in \mathbb{N},\\
		\sum_{j=1}^k \rho^{k-j} Y_{2j}, &\quad n=2k, k \in \mathbb{N}.\\
	\end{cases}
\]
In particular, $(X_{2n})$ and $(X_{2n-1})$ define two independent sequences with the same law as $(Z_n)_{n \geq 1}$ given by $Z_n = \rho Z_{n-1} + Y_n$. Hence, 
\begin{equation}\label{eq:AR_square}
\begin{array}{rcl}
 \pr{}{\sup_{n=1,\dots,2N} X_n \leq x}   & = & \left(\pr{}{\sup_{n=1,\dots,N} Z_n \leq x}\right)^2, \\
 \pr{}{\sup_{n=1,\dots,2N-1} X_n \leq x} & = & \pr{}{\sup_{n=1,\dots,N} Z_n \leq x}\pr{}{\sup_{n=1,\dots,N-1} Z_n \leq x} .
\end{array}
\end{equation}
In particular, the behaviour of the survival probability can be determined by the survival probabilities of AR($1$)-processes. If $\rho = 1$, $X$ defines two indpendent random walk, so we immediately obtain the following lemma:
\begin{lemma}\label{lem:RW_squared}
   Let $\E{}{Y_1} = 0$, $\E{}{Y_1^2} \in (0,\infty)$. $X_{n} = X_{n-2} + Y_{n}$. Then for any $x \geq 0$, there is a constant $c(x)$ such that
\[
   \pr{}{\sup_{n=1,\dots,N} X_n \leq x} \sim c(x) N^{-1}, \quad N \to \infty.
\]
\end{lemma}
\begin{proof}
   By the preceding discussion, $(X_{2n})$ and $(X_{2n-1})$ define two independent centered random walks with finite variance that have the same law. It is then well known that $\pr{}{\sup_{n=1,\dots,N} \sum_{k=1}^n Y_k \leq x} \sim d(x) N^{-1/2}$. The assertion follows in view of \eqref{eq:AR_square}.
\end{proof}
\begin{remark}
By the same reasoning, if $X_n = X_{n-p} + Y_n$ ($p \geq 1$), we have that $p_N(x) \sim c(x) N^{-p/2}$ for any $x \geq 0$ if $\E{}{Y_1} = 0$ and $\E{}{Y_1^2} \in (0,\infty)$. 
\end{remark}
\section{A positive limit}\label{sec:pos_lim}
We now turn to the case that the survival probability converges to a positive limit, i.e.\ $p_N(x) \to p_\infty(x) > 0$ as $N \to \infty$, implying that the process $(X_n)_{n \geq 1}$ stays below $x$ at all times with positive probability. If $X_n = \sum_{k=1}^n c_{n-k} Y_k$, one would expect that this happens if $0 < c_n \to \infty$ and $c_n - c_{n-1} \to \infty$. Indeed, if $c_n$ is very large compared to $c_k$ for $k \le n-1$, then $Y_1 \leq -\delta$ for some $\delta > 0$ implies that $X_n \leq - \delta c_n + \sum_{k=2}^n c_{n-k} Y_k$, and one expects that the expression on the r.h.s.\ stays below a fixed barrier with high probability. In fact, we can transform this idea directly into a proof.
\begin{prop}\label{prop:AR_1_coeff_greater_one}
   Let $(\alpha_n)_{n \geq 0}$ denote a sequence of positive numbers. Let $\rho > 1$ and assume that $\pr{}{Y_1 < 0} > 0$ and $\pr{}{Y_1 \geq x} \precsim (\log x)^{-\alpha}$ as $x \to \infty$ for some $\alpha > 1$. Let $X_n := \sum_{k=1}^n \rho^{n-k} \alpha_{n-k} Y_k$. 
\begin{enumerate}
   \item If $(\alpha_n)_{n \geq 0}$ is nondecreasing, there is a constant $c > 0$ such that 
\[
   \pr{}{\bigcap_{n=1}^\infty \set{X_n \leq - c \alpha_{n-1} \rho^{n-1}} } > 0.
\]
\item If $0 < l \leq \alpha_n \leq u < \infty$ for all $n \geq 0$, there is a constant $c > 0$ such that
\[
   \pr{}{\bigcap_{n=1}^\infty \set{X_n \leq - c \rho^{n-1}} } > 0.
\]
\end{enumerate}
\end{prop}
\begin{proof}
We first prove part 1. Let $\delta > 0$ such that $\pr{}{Y_1 \leq -\delta} > 0$. Let $\beta > 0$ denote a sequence of positive numbers with $\beta \sum_{k=1}^\infty k^{-2} \leq \delta / 2$. Then
\[
   A_N := \set{Y_1 \leq -\delta} \cap \bigcap_{n=2}^N \set{Y_n \leq \rho^{n-1} \beta n^{-2}} \subseteq \bigcap_{n=1}^N \set{X_n \leq - \delta \alpha_{n-1} \rho^{n-1} / 2}
\]
Indeed, since $(\alpha_n)$ is nondecreasing, the event $A_N$ implies that $X_1 = \alpha_0 Y_1 \leq -\alpha_0 \delta$ and for all $n = 2,\dots,N$ that 
\begin{align*}
   X_n & = \rho^{n-1} \alpha_{n-1} Y_1 + \sum_{k=2}^n \rho^{n-k} \alpha_{n-k} Y_k \leq - \delta \alpha_{n-1} \rho^{n-1} +  \rho^{n-1} \sum_{k=2}^n \alpha_{n-k} \beta k^{-2} \\
&\leq - \delta \alpha_{n-1} \rho^{n-1} + \rho^{n-1} \alpha_{n-1} \beta \sum_{k=1}^\infty  k^{-2} =  \alpha_{n-1} \rho^{n-1} \left( \beta \sum_{k=1}^\infty k^{-2} - \delta \right) \leq - \delta  \alpha_{n-1} \rho^{n-1} / 2.
\end{align*}
Finally, in view of the assumption on the tail behaviour of $Y_1$, it is not hard to show that
\begin{align*}
   \lim_{N \to \infty} \pr{}{A_N} = \pr{}{Y_1 \leq -\delta} \lim_{N \to \infty} \prod_{n=2}^N \left(1 - \pr{}{Y_1 > \beta \rho^{n-1} n^{-2}} \right) > 0. 
\end{align*}
The proof of part 2 is very similar. Let $A_N$ be defined as above. Then, using the bounds on $(\alpha_n)$, we get for $n = 2,\dots,N$ that 
\begin{align*}
   X_n &\leq - \delta \alpha_{n-1} \rho^{n-1} +  \rho^{n-1} \sum_{k=2}^n \alpha_{n-k} \beta k^{-2} \leq - \delta l \rho^{n-1} + \rho^{n-1} u \beta \sum_{k=1}^\infty  k^{-2} \\
&=  \rho^{n-1} \left( \beta u \sum_{k=1}^\infty k^{-2} - \delta l \right).
\end{align*}
For $\beta > 0$ sufficiently small, this implies that $X_n \leq - (\delta \, l /2) \, \rho^{n-1}$ for all $n=2,\dots,N$.
\end{proof}
We can now prove Theorem~\ref{thm:pos_lim_summary} showing that the survival probability converges to a positive constant if $X$ is AR(2) with $(a_1,a_2) \in C$ (cf.\ Figure \ref{fig:R_2_decomp}) under mild conditions.  
\begin{proof}(of Theorem~\ref{thm:pos_lim_summary})
Let $(a_1,a_2) \in C$. Assume first that $a_1 > 0$ and $a_2 \in \mathbb{R}$ such that $a_1^2 + 4a_2 > 0$. Moreover, assume that either $a_1 \geq 2$ or $a_1+a_2 > 1$ if $a_1 < 2$. Recall from \eqref{eq:sol_diff_eq} that $c_n = s_1^{n+1}/h  - s_2^{n+1}/h$ where $h > 0$ since $a_1^2 + 4a_2 > 0$. Note that $s_1 = (a_1+h)/2 > 1$ if and only if either $a_1 \geq 2$ or if $a_1+a_2 > 1$ in case $a_1 < 2$. Moreover $\abs{s_2} < s_1$ if and only if $a_1> 0$ and $h > 0$. Hence, in view of our assumptions, we have that $c_n = s_1^n s_1/h \, (1 - (s_2/s_1)^{n+1}) =: s_1^n \alpha_n \geq 0$ for all $n$. Note that $\alpha_n \to s_1/h > 0$. Hence, the assertion follows by part 2 of Proposition~\ref{prop:AR_1_coeff_greater_one}.\\
If $a_1^2+4a_2 = 0$ and $a_1 > 2$, $c_n = (a_1/2)^n (n+1)$ by \eqref{eq:sol_diff_eq}. Hence, the result follows from part 1 of Proposition \ref{prop:AR_1_coeff_greater_one} with $\rho = a_1 / 2 > 1$ and $\alpha_n = n+1$. \\ 
Finally, if $a_1 = 0$ and $a_2 > 1$, the claim follows in view of \eqref{eq:AR_square} and Proposition~\ref{prop:AR_1_coeff_greater_one}.
\end{proof}

\bibliographystyle{abbrvnat}
\bibliography{bib-AR}

\end{document}